\documentclass[reqno,10pt,a4paper]{amsart}

\numberwithin{equation}{section}

\usepackage{amsmath}

\usepackage{esint} % permette di fare l'integrale tagliato

\usepackage{amsthm}

\usepackage{epsfig}

\usepackage{psfrag}

\usepackage{graphicx}

\usepackage{graphpap,latexsym,epsf}

\usepackage{color}

\usepackage{amssymb,mathrsfs,enumerate}

\usepackage{endnotes}

\usepackage{calligra}
\usepackage{a4wide}

\usepackage{accents}
\usepackage{bbm}
\usepackage{dsfont}

\definecolor{citegreen}{rgb}{0,0.6,0}

\definecolor{refred}{rgb}{0.8,0,0}

\usepackage{enumitem}

\usepackage{mathtools}
\mathtoolsset{showonlyrefs}

\usepackage{palatino}
\usepackage[colorlinks, citecolor=citegreen, linkcolor=refred]{hyperref}

\newcommand{\R}{\mathbb{R}}

\newcommand{\N}{\mathbb{N}}

\newcommand{\Z}{\mathbb{Z}}

\newcommand{\SSS}{\mathbb{S}}

\def\HHH{{\rm H}}

\newcommand{\pa}{\partial}

\newcommand{\Ric}{{\rm Ric}}

\newcommand{\D}{{\rm D}}

\newcommand{\na}{\nabla}

\def\ringg#1{\accentset{\circ}{#1}}

\mathchardef\emptyset="001F

\newcommand{\xx}{\hspace{.07em}}

\newcommand{\vertl}{\vert\hspace{.07em}}
\newcommand{\vertr}{\hspace{.07em}\vert}

\definecolor{vgreen}{rgb}{0.1,0.5,0.2}

\definecolor{viola}{RGB}{85,26,139}

\theoremstyle{plain}

\newtheorem{theorem}{Theorem}[section]

\newtheorem{proposition}[theorem]{Proposition}
\newtheorem{lemma}[theorem]{Lemma}

\theoremstyle{definition}
\newtheorem{definition}[theorem]{Definition}

\theoremstyle{remark}
\newtheorem{remark}[theorem]{Remark}

\definecolor{byzantium}{rgb}{0.44, 0.16, 0.39}

\definecolor{amber}{rgb}{1.0, 0.75, 0.0}

\definecolor{darkmagenta}{rgb}{0.55, 0.0, 0.55}

\definecolor{fuzzywuzzy}{rgb}{0.8, 0.4, 0.4}

\definecolor{brown}{rgb}{0.2, 0.08, 0.08}

\definecolor{arancio}{rgb}{1.0, 0.13, 0.0}

\begin{document}
\title{\texorpdfstring{$X$--ADM}{} Mass and \texorpdfstring{$X$--Positive Mass Theorem}{}}
\date{\today}
\author[Carlo Mantegazza]{Carlo Mantegazza}
\address[Carlo Mantegazza]{Department of Mathematics and Applications ``Renato Caccioppoli'', Universit\`a di Napoli Federico II, Italy}
\email[C. Mantegazza]{carlo.mantegazza@unina.it}

\author[Francesca Oronzio]{Francesca Oronzio}
\address[Francesca Oronzio]{Scuola Superiore Meridionale, Napoli, Italy}
\email[F. Oronzio]{f.oronzio@ssmeridionale.it}

\begin{abstract}
For a given admissible vector field $X$, we define a geometric quantity for asymptotically flat $3$--manifolds, called {\em $X$--ADM mass} and we establish a relative {\em positive mass theorem} via a monotonicity formula along the level sets of a suitable Green's function. Under different assumptions on $X$, we obtain generalizations of the ``classical'' positive mass theorem, like the one for weighted manifolds and the one ``with charge'' under some topological restrictions. Finally, we also discuss the rigidity cases.
\end{abstract}

\maketitle

\bigskip



%%%%%%%%%%%%%%%%%%%%%%%%%%%%%%%%%%%%%%%%%%%%%
%%%%%%%%%%%%%%%%%%%%%%%%%%%%%%%%%%%%%%%%%%%%%

\section{\texorpdfstring{$X$}{X}--ADM mass and the main theorem}\label{}
Let $(M, g)$ be a Riemannian manifold. We denote by $\mathrm{Ric}$
and $\mathrm{R}$ the Ricci tensor and the scalar curvature, respectively. The symbol $\nabla$ denotes the Levi--Civita connection, while the associated Laplace--Beltrami operator is given by $\Delta=\mathrm{tr}\nabla^2$. The convention of summing over the repeated indices is always adopted unless otherwise stated.

In order to state precisely our main result, we recall the definition of an asymptotically flat manifold and introduce the notion of $X$--ADM mass.

\begin{definition}\label{defAFman}
A $3$--dimensional Riemannian manifold $(M,g)$ (with or without boundary) is said to be {\em asymptotically flat} if there exists a closed and bounded subset $K$ such that $M\setminus K$ is diffeomorphic to $\R^{3}$ minus a closed ball $\overline{B}_{r}(0)$ by means of a coordinate map $\varphi:M\setminus K\to\R^3\setminus\overline{B}_{r}(0)$ such that
\begin{equation}\label{eq1}
g_{ij}=\delta_{ij} + O_{2}(\vert x \vert ^{-\tau})\,,
\end{equation}
for some $\tau>1/2$ and all $i,j\in \{1,2,3\}$. The pair $(M\setminus K, \varphi)$ will be called an {\em asymptotically flat coordinate chart} of $(M,g)$ and the real number $\tau$ will be called the {\em order of decay} (briefly, {\em the order}) of the metric $g$ in such asymptotically flat coordinate chart.
\end{definition}

Above, we adopted the Landau big--$O$ convention (which can be easily adapted to manifolds that are diffeomorphic, outside a closed and bounded subset, to $\R^{3}$ minus a closed ball with center at the origin). We briefly recall it. Let $f$ be a smooth real--valued function defined outside a compact set of $\R^3$, and let $\tau\in \R$. We write $f=O_k(|\xx x\xx |^{-\tau})$ if there exists a constant $C>0$ such that $\vert \xx \partial^{\alpha}f(x)\xx\vert \leqslant C\vert \xx x\xx\vert ^{-|\xx \alpha\xx|-\tau}$ for every $x\in \R^3\setminus B_R(O)$, when $R>0$ is sufficiently large, for every multi--index $\alpha$ with $0 \leqslant |\xx \alpha\xx| \leqslant k$.

\begin{definition}[$X$--ADM mass]\label{defADMmass}
Let $(M^{3},g)$ be an asymptotically flat manifold. A smooth vector field $X$ on $M$ is said to be {\em admissible} if there exists an asymptotically flat coordinate chart $(M\setminus K, \varphi)$ such that $X^i=O_1(\vert x\vert^{-1-\tau_0})$, for every $i=1,2,3$ and some $\tau_0>1/2$.\\
Given an admissible smooth vector field $X$, we will assume that $\mathrm{R}+2\,\mathrm{div}(X)\in L^1(M,g)$ so that the limit
\begin{equation}\label{formXADMmass}
m_{X}=\frac{1}{16\pi}\lim_{r\to +\infty}\int\limits_{\{\vert x \vert\,=\,r\}}\!\!\!(\partial_{j}g_{ij}-\partial_{i}g_{jj}+2X^i)\frac{x^{i}}{\vert x \vert}\,d\sigma_{\mathrm{eucl}}, 
\end{equation}
that we call {\em $X$--ADM mass} of $(M,g)$, exists and is finite. Indeed, as for the by--now ``classical'' {\em ADM mass}, the existence and finiteness of the above limit, as well as its independence of the asymptotically flat coordinate chart, can be established by following the same arguments as Bartnik~\cite{Bartnik} or Chru\'sciel~\cite{Chrusciel}. 
\end{definition}

The aim of this paper is to prove the following $3$--dimensional {\em generalized positive mass theorem}, relative to the $X$--ADM mass.
 
\begin{theorem}[$X$--Positive mass theorem]\label{GenPMT}
Let $(M,g)$ be an orientable, connected and complete, asymptotically flat $3$--manifold and let $X$ be an admissible smooth vector field on $M$. Assume that:
\begin{itemize}
\item[$(a)$] $\mathrm{R}+2\,\mathrm{div}(X)\in L^{1}(M,g)$;
\item[$(b)$] $\mathrm{R}_X^{(k)}=\mathrm{R}+2\,\mathrm{div}(X)-(1+1/k)\, \vert X\vert^2\, \geqslant\,0$, for some $k\in \R\setminus (-2,0]$;
\item[$(c)$] the second integral homology group $H_2(M;\Z)$ does not contain any ``spherical'' class.
\end{itemize}
Then,
\begin{equation}\label{GenPMI-main}
m_{X} \geqslant 0\,.
\end{equation}
Furthermore, the following statements are true.
\begin{enumerate}
\item If $\mathrm{R}_X^{(k)} \geqslant 0$ with $k\in \R\setminus [-2,0]$ and $m_{X}=0$ (the equality case), then $X$ vanishes on $M$ and $(M,g)$ is isometric to $(\R^3,g_{\mathrm{eucl}})$.
\item If $\mathrm{R}_X^{(-2)} \geqslant 0$ and $m_{X}=0$, then $X$ is a gradient of a smooth function and $(M,g)$ is conformally isometric to $(\R^3,g_{\mathrm{eucl}})$.
\end{enumerate}
\end{theorem}

The definition of $X$--ADM mass and this associated positive mass theorem $X$--PMT provide a unified framework that encompasses several known results as special cases. Specifically:

\begin{enumerate}
\item[$\circ$] If $X^i=O_1(\vert x\vert^{-1-2\tau_0})$, for every $i\in \{1,2,3\}$ and some $\tau_0>1/2$, the $X$--ADM mass coincides with the ``classical'' ADM mass
\begin{equation}\label{formXADMmass0}
m_{\text{ADM}}=m_0=\frac{1}{16\pi}\lim_{r\to +\infty}\int\limits_{\{\vert x \vert\,=\,r\}}\!\!\!(\partial_{j}g_{ij}-\partial_{i}g_{jj})\frac{x^{i}}{\vert x \vert}\,d\sigma_{\mathrm{eucl}} 
\end{equation}
By choosing $X$ to be identically zero (which is obviously admissible), the $X$--PMT then reduces to the ``classical'' positive mass theorem (see~\cite{SchYau79, Witten}) in dimension 3 and with $H_2(M;\Z)$ without spherical classes, while, if $X$ is non--zero, the $X$--PMT still yields the ``classical'' positive mass inequality under a more general assumption on the scalar curvature. 
Hirsch, Miao, and Tsang~\cite[Corollary 1.3]{HirshMiaoTsang} obtained a similar result allowing the presence of ``corners'' in the manifolds. In the smooth case, however, our theorem encompasses a broader class of manifolds.

\item[$\circ$] In~\cite{BaOz2002}, Baldauf and Ozuch introduced {\em the weighted mass} $m_f(g)$ for an asymptotically flat weighted manifold $(M^n, g, f)$. In dimension $3$, the $\nabla f$--ADM mass is actually equal to $m_f(g)/16\pi$, therefore, our Theorem~\ref{GenPMT} includes the {\em weighted positive mass theorem}~\cite[Theorem 1.6]{LaLoSa} of Law, Lopez, and Santiago (under the above topological assumption on $H_2(M;\Z)$).

\item[$\circ$] Theorem~\ref{GenPMT} is also related to the following positive mass theorem {\em with charge}. Let $\mathcal{E}$ be an admissible smooth vector field, representing an electric field. If $\mathrm{R}$ and $\mathrm{div}(\mathcal{E})$ belong to $L^1(M,g)$ and the inequality $\mathrm{R}-4\,\vert \mathrm{div}(\mathcal{E})\vert-2\vert \mathcal{E}\vert^2 \geqslant 0$ holds, then $m_{\mathrm{ADM}} \geqslant \vert \mathcal{Q}\vert$, where $\mathcal{Q}$ is the ``total charge'' given by
$$
\mathcal{Q}=(1/4\pi)\,\int_{\{\vert x \vert\,=\,r\}}\!\mathcal{E}^i (x^{i}/\vert x \vert)\,d\sigma_{\mathrm{eucl}}.
$$
This kind of positive mass theorem was established under quite general assumptions by Chru\'sciel, Reall and Tod in~\cite{ChReTo2006} (for recent developments on the subject, see~\cite{Rau25} and references therein). We notice that (under the above topological assumption on $H_2(M;\Z)$) it follows from Theorem~\ref{GenPMT}. Indeed, we may assume without loss of generality that $\mathcal{Q}\geqslant 0$, then since the inequality $\mathrm{R}-4\,\vert \mathrm{div}(\mathcal{E})\vert-2\vert \mathcal{E}\vert^2 \geqslant 0$ implies $\mathrm{R}_{X}^{(-2)}\geqslant 0$ with $X=-2 \mathcal{E}$, by applying Theorem~\ref{GenPMT} (with $X=-2 \mathcal{E}$), we get the inequality $m_{\mathrm{ADM}} \geqslant \vert \mathcal{Q}\vert$.
\end{enumerate}

The $X$--PMT is established through a monotonicity formula holding along the regular level sets of the minimal positive Green's function for the operator $\mathcal{L}_X=\Delta -\,(1/2)\xx\nabla_X$ with a pole at some point of $M$.

In general, monotonicity formulas play an important role in geometric analysis.
In comparison geometry, classical examples include the monotonicity formula for
minimal submanifolds and the Bishop--Gromov volume
comparison theorem and in the context of geometric flows, the
Huisken monotonicity formula for the mean curvature
flow~\cite{Hui_1990}, the Perelman entropy formula for the Ricci
flow~\cite{Per} and the Geroch monotonicity of the Hawking mass along
the inverse mean curvature flow~\cite{ger73,HI}. In a series of works~\cite{Colding_Acta, Col_Min_2,Col_Min_3}, 
Colding and Minicozzi obtained some monotonicity
formulas along the level sets of the minimal positive Green function for the Laplacian operator with a pole in non--parabolic Riemannian manifolds
with nonnegative Ricci curvature. Using such monotonicity
formulas, they showed the uniqueness of tangent cones for Einstein
manifolds~\cite{Col_Min_4}. The one--parameter family of monotonicity formulas introduced in~\cite{Col_Min_3} has been extended for weighted manifolds with nonnegative Bakry--\'{E}mery Ricci
curvature in~\cite{Songetall}, as well as for the case of positive Ricci curvature in~\cite{Manea}. In the context of potential theory, similar monotonicity formulas to those of Colding and Minicozzi were used in~\cite{Ago_Fog_Maz_1} to obtain new Willmore--type geometrical inequalities and generalized in~\cite{BFM} to prove an optimal version of the Minkowski inequality, in non--parabolic Riemannian manifolds with nonnegative Ricci curvature

In recent years, monotonicity formulas in the context of a ``level set approach'' have been intensively applied to the study of $3$--manifolds with nonnegative scalar curvature. A significant step in this direction was made by Stern in~\cite{Stern}, where an integral inequality linking the scalar curvature of a closed $3$--manifold to harmonic maps into $\SSS^1$ was derived using the Bochner formula, the traced Gauss equation, and the Gauss--Bonnet theorem. Since then, new monotonicity formulas holding along level sets of harmonic and $p$--harmonic functions have led to new results for $3$--manifolds with nonnegative scalar curvature, in particular, in the asymptotically flat setting. For instance, a new proof of the positive mass theorem was obtained in~\cite{Ago_Maz_Oro_2} and simpler proofs of the Riemannian Penrose inequality (without rigidity) were found in~\cite{AMMO, Hir_Miao_Tam} by exploiting $p$--harmonic functions. 
New results involving the connections ADM--mass/$p$--capacity and area/$p$--capacity were established in~\cite{Miao2022, Or_1} via harmonic functions and in~\cite{MazYao_1, XiaYinZhou} via $p$--harmonic functions.
Outside of the asymptotically flat setting, gradient integral estimates were found in~\cite{ChChLeTs, Col_Min_5, MuntWang, Or_2} for $3$--manifolds with nonnegative scalar curvature and a new positive mass theorem in asymptotically hyperbolic three--manifolds in~\cite{KrOrPi} was recently proved using the Green’s functions with a pole for the Laplacian operator.
Among the many other related developments of the positive mass theorem, which was proved by Schoen and Yau~\cite{SchYau79, SchYau81} in the case $ 3\leqslant n\leqslant 7$ and by Witten~\cite{Witten} in the spin case, there are its generalizations to weighted manifolds. These generalizations were recently established by Baldauf--Ozuch in~\cite{BaOz2002} in the spin case, by Chu--Zhu~\cite{ChuZhu} in the case $3\leqslant n\leqslant 7$, and by Law--Lopez--Santiago~\cite{LaLoSa} for smooth metric measure spaces. The analogues of the Ricci and scalar curvatures in a smooth metric measure space are the $m$--Bakry--\'{E}mery Ricci and scalar curvatures for $m\neq 0$. A natural generalization of the $m$--Bakry--\'{E}mery Ricci curvature $\mathrm{Ric}+\nabla df -(1/m)df\otimes df$ with a given weight $f$ of class $C^{\infty}$ is the $m$--Bakry--\'{E}mery Ricci curvature $\mathrm{Ric}+(1/2)L_{X}g-(1/m)X^{\flat}\otimes X^{\flat}$ with a given smooth vector field $X$. Taking into account the corresponding scalar curvature, the formulation of $X$--PMT arises. Finally, the use of the drift Laplacian is not new and already appeared in~\cite{Bra_Hir_Kaz_Khu_Zha_2021}. In that paper, the authors considered solutions that are asymptotic to linear coordinate functions; therefore, the analogy between the present work and~\cite{Bra_Hir_Kaz_Khu_Zha_2021} is the same as the one between~\cite{Ago_Maz_Oro_2} and~\cite{bray3}. However, there are three main differences that we now describe. First, we allow the smooth vector field to have non--zero divergence, while in their work, Bray, Hirsch, Kazaras, Khuri and Zhang considered only divergence--free vector fields. Secondly, their setting allows asymptotically cylindrical ends for the manifolds. Thirdly, we also allow other (stronger) hypotheses on the scalar curvature in order to have a complete picture of the rigidity case (without asymptotically cylindrical ends).

\smallskip

As mentioned above, the key ingredient of Theorem~\ref{GenPMT} is the following monotonicity result.

\begin{theorem}\label{GenPMTmon}
Let $(M,g)$ be a complete, non--compact, orientable, $3$--dimensional Riemannian manifold, and let $X$ be a smooth vector field on $M$. Suppose that:
\begin{itemize}
\item[$(a)$] $\mathrm{R}_X^{(k)}=\mathrm{R}+2\,\mathrm{div}(X)-(1+1/k)\, \vert X\vert^2\, \geqslant\,0$, for some $k\in \R\setminus (-2,0]$;
\item[$(b)$] the second integral homology group $H_2(M;\Z)$ does not contain any spherical class.
\end{itemize}
Assuming that there exists the minimal positive Green's function $\mathcal{G}_o$ for $\mathcal{L}_X=\Delta -\,(1/2)\xx\nabla_X$ with a pole at some point $o\in M$ and vanishing ``at infinity'', we consider the function
 \begin{equation}\label{definition_u}
 u\,=\,1-4\pi \mathcal{G}_o\,,
 \end{equation}
 and let $F:(0, + \infty)\to \R$ be defined as follows:
 \begin{equation}
 F(t)\ = \ 4\pi t \ +\ \,t^3\!\!\!\int\limits_{\{u=1-1/t\}}\!\!\!\!\! \vert \nabla u \vert^{2} \ d\mathcal{H}^2 \, \, - \ t^2 \!\!\! \int\limits_{\{u=1-1/t\}}\!\!\!\!\!\vert \nabla u \vert\, \mathrm{H} \ d\mathcal{H}^2 
 \, \, + \ t^2 \!\!\! \int\limits_{\{u=1-1/t\}}\!\!\!\!\! g(X,\nabla u) \ d\mathcal{H}^2 \,,\label{eq0}
 \end{equation}
 where $\mathrm{H}$ is the mean curvature on the regular part of the level set $\{u=1-1/t \}\setminus \{\vert \nabla u \vert=0\}$ computed with respect to the unit normal vector field $\nu={\nabla u}/{\vert \nabla u \vert}$.
 Then, we have 
 \begin{equation}\label{monotonicity}
 0 < s \leqslant t < +\infty \quad\implies\quad F(s) \, \leqslant \, F(t) \, ,
 \end{equation}
 provided $1- 1/s$ and $1- 1/t$ are regular values of $u$.
\end{theorem}

Roughly speaking, Theorem~\ref{GenPMT} will follow from Theorem~\ref{GenPMTmon} by taking the limits of $F(t)$ as $t\to0^+$ and as $t\to+\infty$.

The paper is organized as follows: In Section~\ref{Greensection}, we show the existence and the needed properties of the minimal positive Green's function with a pole for the operator $\mathcal{L}_X=\Delta -\,(1/2)\xx\nabla_X$ when $M$ is asymptotically flat. In Section~\ref{monotonicity formula section}, we prove the monotonicity result above, Theorem~\ref{GenPMTmon}. Then, finally, in Section~\ref{XPMTsection}, we obtain the $X$--positive mass theorem, Theorem~\ref{GenPMT}.

\section{The minimal positive Green's function for \texorpdfstring{$\mathcal{L}_X$}{} with a pole}\label{Greensection}

The aim of this section is to establish the existence of the minimal positive Green's function $\mathcal{G}_o$ for the operator $\mathcal{L}_X=\Delta -\,(1/2)\xx\nabla_X$ with a pole at some point $o\in M$, to describe its asymptotic behavior near the pole and at infinity in a complete asymptotically flat $3$--manifold $(M,g)$, under the assumption that $X$ is an admissible smooth vector field on $M$. 

\subsection{Existence and asymptotic behavior of the minimal positive Green's function \texorpdfstring{$\mathcal{G}_o$}{} for the operator \texorpdfstring{$\mathcal{L}_X$}{} near the pole.}\label{subsecexistenceandasymptoticnearthepole}

We adopt a step--by--step approach. The second and third steps are in the same spirit of~\cite[Section 3]{KrOrPi}.

\smallskip

\noindent \textbf{Step~1.} {\em Let $\Omega$ be an arbitrary bounded and connected open set with smooth boundary such that $o\in \Omega$.
The operator $\mathcal{L}_X$ admits a unique positive minimal Green's function $\mathcal{G}^{\Omega}_o$ in $\Omega$ with pole at $o$, which is the Dirichlet Green's function with pole $o$ of $\mathcal{L}_X$ in $\Omega$. Moreover, for any normal coordinate system $(x^1,x^2,x^3)$ centered at $o$ and defined on an open ball $B_{r_o}(o)$ with $\overline{B}_{r_0}(o)\subseteq\Omega$, there exist constants $b_{i_1}^{(0)}$, $a^{(1)}$, $c_{i_1 i_2}^{(1)}$, $e_{i_1 i_2 i_3 i_4 }^{(1)}$, $b_{i_1}^{(2)}$, $d_{i_1 i_2 i_3 }^{(2)}$, $f_{i_1 i_2 i_3 i_4 i_5 }^{(2)}$, $a^{(3)}$, $c_{i_1 i_2}^{(3)}$, $e_{i_1 i_2 i_3 i_4 }^{(3)}$, $h_{i_1 i_2 i_3 i_4 i_5 i_6}^{(3)}$ and $\,l_{i_1 i_2 i_3 i_4 i_5 i_6 i_7 i_8}^{(3)}$, depending on the coefficients of the metric $g$ and the vector field $X$, and there exists a function $f\in C^{2}(\overline{\Omega})$ such that
\begin{align}
4\pi \mathcal{G}^{\Omega}_o&=\frac{1}{\vert x\vert}+b_{i_1}^{(0)}\frac{x^{i_1}}{\vert x\vert}+\vert x\vert\Big[a^{(1)}+c_{i_1 i_2}^{(1)}\frac{x^{i_1} x^{i_2}}{\vert x\vert^2}+e_{i_1 i_2 i_3 i_4 }^{(1)}\frac{x^{i_1}x^{i_2} x^{i_3} x^{i_4}}{\vert x\vert^4}\Big]\\
&\quad \ +\vert x\vert^2\Big[b_{i_1}^{(2)}\,\frac{x^{i_1}}{\vert x\vert}+d_{i_1 i_2 i_3 }^{(2)}\frac{x^{i_1}x^{i_2} x^{i_3} }{\vert x\vert^3}+f_{i_1 i_2 i_3 i_4 i_5 }^{(2)}\frac{x^{i_1}x^{i_2} x^{i_3} x^{i_4} x^{i_5}}{\vert x\vert^5}\Big]\\
&\quad \ +\vert x\vert^3\Big[a^{(3)}+c_{i_1 i_2}^{(3)}\frac{x^{i_1} x^{i_2}}{\vert x\vert^2}+e_{i_1 i_2 i_3 i_4 }^{(3)}\frac{x^{i_1}x^{i_2} x^{i_3} x^{i_4}}{\vert x\vert^4}+h_{i_1 i_2 i_3 i_4 i_5 i_6}^{(3)}\frac{x^{i_1}x^{i_2} x^{i_3} x^{i_4} x^{i_5} x^{i_6}}{\vert x\vert^6}\\
&\quad\quad\quad \quad\ +l_{i_1 i_2 i_3 i_4 i_5 i_6 i_7 i_8}^{(3)}\frac{x^{i_1}x^{i_2} x^{i_3} x^{i_4} x^{i_5} x^{i_6} x^{i_7} x^{i_8}}{\vert x\vert^8}\Big] + f,\label{feq10}
\end{align}
in $B_{r_o/4}(o)$.}\\ 
On the punctured open ball $B_{r_o}(o)\setminus\{o\}$ we consider the function
\begin{align}
&w=\frac{1}{\vert x\vert}+b_{i_1}^{(0)}\frac{x^{i_1}}{\vert x\vert}+\vert x\vert\Big(a^{(1)}+c_{i_1 i_2}^{(1)}\frac{x^{i_1} x^{i_2}}{\vert x\vert^2}+e_{i_1 \cdots i_4 }^{(1)}\frac{x^{i_1}\cdots x^{i_4}}{\vert x\vert^4}\Big)\\
&\quad\quad+\vert x\vert^2\Big(b_{i_1}^{(2)}\,\frac{x^{i_1}}{\vert x\vert}+d_{i_1 i_2 i_3 }^{(2)}\frac{x^{i_1}x^{i_2} x^{i_3} }{\vert x\vert^3}+f_{i_1\cdots i_5 }^{(2)}\frac{x^{i_1}\cdots x^{i_5}}{\vert x\vert^5}\Big)\\
&\quad\quad +\vert x\vert^3\Big(a^{(3)}+c_{i_1 i_2}^{(3)}\frac{x^{i_1} x^{i_2}}{\vert x\vert^2}+e_{i_1\cdots i_4 }^{(3)}\frac{x^{i_1}\cdots x^{i_4}}{\vert x\vert^4}+h_{i_1 \cdots i_6}^{(3)}\frac{x^{i_1}\cdots x^{i_6}}{\vert x\vert^6}
+l_{i_1\cdots i_8}^{(3)}\frac{x^{i_1}\cdots x^{i_8}}{\vert x\vert^8}\Big)\,.\label{eqcar9999}
\end{align}
Expanding the coefficients of the operator $\mathcal{L}_X=\Delta -\,(1/2)\xx\nabla_X$, that is, expanding the functions $X^i$, $g^{ij}$ and $\Gamma_{ij}^k$, it is possible to determine the constants $b_{i_1}^{(0)}$, $a^{(1)}$, $c_{i_1 i_2}^{(1)}$, $e_{i_1 i_2 i_3 i_4 }^{(1)}$, $b_{i_1}^{(2)}$, $d_{i_1 i_2 i_3 }^{(2)}$, $f_{i_1 i_2 i_3 i_4 i_5 }^{(2)}$, $a^{(3)}$, $c_{i_1 i_2}^{(3)}$, $e_{i_1 i_2 i_3 i_4 }^{(3)}$, $h_{i_1 i_2 i_3 i_4 i_5 i_6}^{(3)}$ and $\,l_{i_1 i_2 i_3 i_4 i_5 i_6 i_7 i_8}^{(3)}$ so that the function $w$ satisfies 
\begin{equation}
\mathcal{L}_X w=h\quad \text{ on} \quad B_{r_o}(o)\setminus\{o\} \,,
\end{equation}
where $h$ is a smooth function on the punctured open ball $B_{r_o}(o)\setminus\{o\}$ that admits a $C^1$--extension on $\overline{B}_{r_o}(o)$, still denoted by $h$. See Appendix~\ref{proofconstofw} for a proof.\\\
Then, by Theorem~6.14 in~\cite{GilTr}, there is a unique solution $f$ to the Dirichlet problem 
$$
\mathcal{L}_X f=-h\quad \text{in $\,\,B_{r_o}(o)$}\qquad\text{ with }\qquad f=-w|_{\partial B_{r_o}(o)} \quad \text{on $\,\,\partial B_{r_o}(o)$}
$$
belonging to $C^{2,\alpha}\big(\overline{B}_{r_o}(o)\big)$. We claim that $w+f=4\pi \mathcal{G}^{B_{r_o}(o)}_o$, indeed, by construction, it satisfies $\mathcal{L}_X\big(w+f\big)=-4\pi\delta_o$ on $B_{r_o}(o)$ in the sense of distributions (here, $\delta_o$ denotes the Dirac delta measure at $o$). Indeed, for every $\psi\in C^{\infty}_c(B_{r_o}(o))$, there holds
\begin{align*}
&\int\limits_{B_{r_o}(o)} \!\!\!g\big(\nabla (w+f), \nabla \psi\big)+\frac{1}{2}\, \psi\, g\big(X,\nabla(w+f)\big)\,d\mu\\
&\quad=\lim_{\varepsilon \to 0^+}\int\limits_{B_{r_o}(o)\setminus B_{\varepsilon}(o)} \!\!\!\!g\big(\nabla (w+f), \nabla \psi\big)+\frac{1}{2}\, \psi \,g\big(X,\nabla (w+f) \big)\,d\mu\\
&\quad =\lim_{\varepsilon \to 0^+} \bigg[-\int_{\partial B_{\varepsilon}(o)} \psi\, g\big(\nabla(w+f),\nu\big)\,d\mathcal{H}^2-\!\!\!\int\limits_{B_{r_o}(o)\setminus B_{\varepsilon}(o)}\!\!\!\! \psi\,\mathcal{L}_X(w+f)\,d\mu\bigg]\\
&\quad =-\lim_{\varepsilon \to 0^+} \int_{\partial B_{\varepsilon}(o)} \psi \,g\big(\nabla(w+f),\nu\big)\,d\mathcal{H}^2 =-\lim_{\varepsilon \to 0^+} \int_{\partial B_{\varepsilon}(o)} \psi \,g\big(\nabla w,\nu\big)\,d\mathcal{H}^2=4\pi\psi(o)\,,
\end{align*}
where $\nu$ is the outward pointing normal to $\partial B_{\varepsilon}(o)$ and the last limit is obtained by means of the expression~\eqref{eqcar9999} of $w$ in the normal coordinate system $(x^1,x^2,x^3)$.\\
Notice that $f+w$ is positive in $B_{r_o}(o)\setminus\{o\}$, by the maximum principle. Moreover, for every positive Green's function $\widetilde{G}^{B_{r_o}(o)}_o$ of $\mathcal{L}_X$ in $B_{r_0}(o)$ with pole at $o$, since $\mathcal{L}_X\big(4\pi\widetilde{G}^{B_{r_o}(o)}_o\!\!\!-w-f\big)=0$ in the sense of distributions, the function $4\pi\widetilde{G}^{B_{r_o}(o)}_o\!\!\!-w-f$ can be extended smoothly to the whole $B_{r_0}(o)$ and again by the maximum principle, there holds
$$
\inf_{B_{r_o}(o)}\big(4\pi\widetilde{G}^{B_{r_o}(o)}_o\!\!\!-w-f\big)=\inf_{\partial B_{r_o}(o)}\big(4\pi\widetilde{G}^{B_{r_o}(o)}_o\!\!\!-w-f\big)=4\pi\inf_{\partial B_{r_o}(o)}\widetilde{G}^{B_{r_o}(o)}_o \geqslant 0\,,
$$
hence, the claim follows.\\
Let us now consider the general case of a connected bounded open set $\Omega$ of $M$ with smooth boundary and such that $o\in\Omega$. In this case, we choose a smooth cut--off function $\psi$ which is identically equal to $1$ in $B_{r_o/4}(o)$ and vanishes in $M\setminus \overline{B}_{r_o/2}(o)$. Then, there is a unique solution $f$ to the Dirichlet problem 
$$
\mathcal{L}_Xf=-\psi h-w\mathcal{L}_{X}\psi -2g(\nabla w,\nabla \psi)\quad \text{in $\Omega$}\qquad\text{ with }\qquad f=0\quad \text{on $\partial \Omega$}
$$
belonging to $C^{2,\alpha}\big(\overline{\Omega}\big)$ and arguing as before, it follows that $4\pi \mathcal{G}^{\Omega}_o=\psi w+f$.

\medskip

\noindent \textbf{Step~2.} {\em Existence of exterior and global barriers.}\\
We fix a normal coordinate system centered at $o\in M$ defined on an open ball $B_{r_o}(o)$ and an asymptotically flat coordinate chart $\big(U,\varphi=(x^1,x^2,x^3)\big)$ of $(M,g)$ with decay rate $\tau>1/2$, as in Definition~\ref{defAFman}.\\
For any $0<\varepsilon<\tau$, we consider the function
$$
\phi_+\,=\, \frac{1}{\vert x \vert}\,-\,\frac{1}{\ \vert x \vert^{1+\varepsilon}}\,.
$$
One can check that
\begin{equation}
\mathcal{L}_X \phi_+\,=\,g^{ij} \partial_{x^i}\partial_{x^j}\phi_+\,-\,g^{ij}\, \Gamma_{ij}^k\partial_{x^k}\phi_+\,-\,(1/2) \, X^i \partial_{x^i}\phi_+\,=\,-\,\varepsilon\,(1+\varepsilon)\vert x \vert^{-3-\varepsilon}+O(\vert x \vert^{-3-\tau})\,,
\end{equation}
as a consequence of assumptions~\eqref{eq1} and $X^i=O(\vert x\vert^{-1-\tau})$.
Then, there exists $R> 1$ such that $\phi_+$ is a positive function satisfying $\mathcal{L}_X \phi_+<0$ in the set $\{\vert x\vert \geqslant R\}$.
Similarly, one can show that, possibly passing to a larger $R>1$, the positive function
$$
\phi_-\,=\, \frac{1}{\vert x \vert}\,+\,\frac{1}{\ \vert x \vert^{1+\varepsilon}}\,.
$$
satisfies $\mathcal{L}_X \phi_->0$ in the set $\{\vert x\vert \geqslant R\}$.\\
Without loss of generality, we can assume that $B_{r_o}(o)\cap \{\vert x\vert \geqslant R\}=\emptyset$, then we consider two nonnegative cut--off functions $\psi_{r_o}$ and $\psi_R$, such that:
\begin{itemize}
\item[$\circ$] $\psi_{r_o}$ is the same cut--off function that we used in Step~1: identically equal to $1$ on $B_{r_o/4}(o)$ and vanishing in $M\setminus \overline{B}_{r_o/2}(o)$;
\item[$\circ$] $\psi_{R}$ is identically equal to $1$ on $\{\vert x\vert> 4R\}$ and vanishing in $M\setminus \{\vert x\vert \geqslant 2R\}$.
\end{itemize}
By means of these cut--off functions, we can define on the whole $M$ the function $$\widetilde{h}_{+}=-\psi_{r_o} h-w\mathcal{L}_{X} \psi_{r_o} -2g(\nabla w,\nabla \psi_{r_o})+\psi_{R} \,\mathcal{L}_X \phi_+\,.$$
By~\cite[Section 4]{ChBruhat}, there exists a unique function $v_{+}\in C^{2}(M)$ that vanishes at infinity and satisfies $\mathcal{L}_X v_{+}=\widetilde{h}_{+}$ in $M$. Now, the function $\widetilde{v}_{+}=\psi_{r_o}w+v_{+}$, is smooth in $M\setminus\{o\}$ and satisfies $\mathcal{L}_{X}\widetilde{v}_{+}=-4\pi\delta_o+\psi_R\,\mathcal{L}_X \phi_+$ in $M$ (distributionally). Thus, $\widetilde{v}_{+}$ is a supersolution of $\mathcal{L}_X$ in $M\setminus \{o\}$ which tends to $+\infty$ getting close to the pole $o$ and vanishes at infinity. By the maximum principle, it follows that $\widetilde{v}_{+}$ is positive in $M\setminus\{o\}$.

\medskip

\noindent \textbf{Step~3.} {\em The operator $\mathcal{L}_X$ admits a unique positive minimal Green's function $\mathcal{G}_o$ in $M$, with a pole at $o$, which vanishes at infinity. Moreover, for any normal coordinate system $(x^1,x^2,x^3)$ centered at $o\in M$ and defined on an open ball $B_{r_o}(o)$, the function $4\pi\mathcal{G}_o$ can be written on $B_{r_o/4}(o)$ as in formula~\eqref{feq10}, with $f\in C^{2}(M)$.}\\
Let $\{\Omega_{j}\}_{j\in \N}$ be a {\em (precompact) exhaustion} of $M$,
i.e. a sequence of bounded, connected open subsets of $M$ with smooth boundaries such that $\overline{B}_{r_o}(o) \subseteq\Omega_1$, $\overline{\Omega}_j\subseteq \Omega_{j+1}$ and $\bigcup_{j\in\N} \Omega_{j}=M$. In Step~1, we showed that the operator $\mathcal{L}_X$ admits a unique positive minimal Green's function $\mathcal{G}^{\Omega_j}_o$ with pole at $o$ which is the Dirichlet Green's function with pole $o$ of $\mathcal{L}_X$ in $\Omega_j$, for every $j\in \N$. By the maximum principle, $\{\mathcal{G}^{\Omega_j}_o\}_{j\in\N}$ is
a strictly increasing sequence of positive functions and each function $\mathcal{G}^{\Omega_j}_o$ satisfies $4\pi \mathcal{G}^{\Omega_j}_o \leqslant \widetilde{v}_{+}$. Therefore, by combining the interior Schauder elliptic estimates with the Ascoli--Arzel\'{a} theorem and a standard diagonalization argument, we obtain that, possibly after passing to a subsequence, the sequence $\mathcal{G}^{\Omega_j}_o$ converges in $C^{k}_{\text{loc}}$ on $M\setminus \{o\}$ to a function $\mathcal{G}_o$, for any fixed $k \geqslant 2$. In particular, the function $\mathcal{G}_o$ satisfies $\mathcal{L}_X \mathcal{G}_o =0$ in $M\setminus \{o\}$. Now, we observe that the following chain of inequalities holds on the punctured open set $\Omega_j\setminus \{o\}$:
\begin{equation}\label{chainineq1}
\psi_{r_o}w+f^{\Omega_j}=4\pi \mathcal{G}^{\Omega_j}_o \leqslant 4\pi \mathcal{G}_o \leqslant \widetilde{v}_{+}=\psi_{r_o}w+v_{+}\,,
\end{equation}
where $w$ is the smooth function on $B_{r_o}(o)\setminus\{o\}$ determined in Step~1, satisfying the equation $\mathcal{L}_X w=h$ on $B_{r_o}(o)\setminus\{o\}$, the function $h$ admits a $C^1$--extension to $\overline{B}_{r_o}(o)$ which we still denote by $h$, $\psi_{r_o}$ is the cut--off function introduced in Step~2 and the functions $f^{\Omega_j}$ are the ones appearing in formula~\eqref{feq10}, relative to the subsets $\Omega_j$.\\
Then, the sequence $\{f^{\Omega_j}\}_{j\geqslant k}$ is not only nondecreasing, but also bounded from above on every $\Omega_k$, for all $k\in \N$. Therefore, by combining once again the interior elliptic estimates with the Ascoli--Arzel\'{a} theorem and a diagonalization argument, possibly passing to a subsequence, the sequence $\{f^{\Omega_j}\}_{j\in\N}$ converges to a function $f$ in $C^{2}_{\text{loc}}(M)$. This implies that the function $f$ satisfies 
$$
\mathcal{L}_X f=-\psi_{r_o} h-w\mathcal{L}_{X}\psi_{r_o} -2g(\nabla w,\nabla \psi_{r_o})
$$
in $M$, since each $f^{\Omega_j}$ solves the same equation in $\Omega^j$. Then, one can check that $\mathcal{L}_X \mathcal{G}_o=-\delta_o$ in $M$, as $4\pi \mathcal{G}_o=\psi_{r_o}w+f$ in $M$ (passing to the limit in the equality on the left--hand side of formula~\eqref{chainineq1}).\\ 
Furthermore, formula~\eqref{chainineq1} implies that $\mathcal{G}_o>0$ on $M\setminus\{o\}$ and that $\mathcal{G}_o$ vanishes at infinity. Finally, if $\widetilde{\mathcal{G}}_o$ is a positive Green's function on $M$, that is, $\widetilde{\mathcal{G}}_o$ is a positive function satisfying $\mathcal{L}_X \widetilde{\mathcal{G}}_o=-\delta_o$ in $M$ in the sense of distributions, the maximum principle yields $\mathcal{G}^{\Omega_j}_o \leqslant \widetilde{\mathcal{G}}_o$ in $\Omega_j\setminus\{o\}$. Passing to the limit, it follows that $ \mathcal{G}_o \leqslant \widetilde{\mathcal{G}}_o$ in $M\setminus\{o\}$.

\subsection{Asymptotic behavior at infinity of the minimal positive Green's function \texorpdfstring{$\mathcal{G}_o$}{} for the operator \texorpdfstring{$\mathcal{L}_X$}{} with pole at \texorpdfstring{$o\in M$}{}.}

The asymptotic behavior at infinity of the minimal positive Green's function $\mathcal{G}_o$ for the operator $\mathcal{L}_X$, with pole at $o$, in a generic asymptotically flat chart follows in the same way as~\cite[Appendix A]{MMT} (see~\cite[Section 4]{hirsch} for the asymptotic expansion of the conformal Green's function). For completeness, we include the proof here.

 Let $\big(U,\varphi=(x^1,x^2,x^3)\big)$ be an asymptotically flat chart of $(M,g)$, as in Step~2. There, we showed that there is $R>1$ such that the functions
$$
\phi_+\,=\, \frac{1}{\vert x \vert}\,-\,\frac{1}{\vert x \vert^{1+\varepsilon}}\quad\quad \text{and}\quad\quad \phi_-\,=\, \frac{1}{\vert x \vert}\,+\,\frac{1}{\vert x \vert^{1+\varepsilon}}
$$
are positive and satisfy $\mathcal{L}_X\phi_+\!<0$ and $\mathcal{L}_X\phi_-\!>0$ in the set $\{\vert x\vert \geqslant R\}$, respectively. Therefore, by Step~3, the maximum principle implies
\begin{equation}\label{eq33}
\Big(\frac{1}{R}\,+\,\frac{1}{R^{1+\varepsilon}}\Big)\Big( \min_{\{\vert x\vert =R\}}\mathcal{G}_o\Big)\phi_- \leqslant \mathcal{G}_o \leqslant \Big(\frac{1}{R}\,-\,\frac{1}{R^{1+\varepsilon}}\Big)\Big( \max_{\{\vert x\vert =R\}}\mathcal{G}_o\Big)\phi_+
\end{equation}
on $\{\vert x\vert \geqslant R\}$. Now, in order to obtain the desired estimates, we consider a cut--off function $\widetilde{\psi}_{R}$ which is identically equal to $1$ on $\{\vert x\vert> 4R\}$ and vanishes in $\R^3\setminus \{\vert x\vert \geqslant 2R\}$. By means of this cut--off function, we define
$$
\widehat{g}=\psi_{R}\,\varphi_*g+(1-\psi_{2R})\,g_{\mathrm{eucl}}\,,
$$
which is a complete asymptotically flat Riemannian metric on $\R^3$, the smooth admissible vector field $\widehat{X}=\psi_R\,\varphi_*X $ on $\R^3$ and $\widehat{\mathcal{G}}_o=\psi_R\,\varphi_*\mathcal{G}_o$, which is a smooth function satisfying
$$
\widehat{\mathcal{L}}_{\widehat{X}}\xx \widehat{\mathcal{G}}_o=\Delta_{\widehat{g}} \xx\widehat{\mathcal{G}}_o-(1/2)\,\nabla^{\widehat{g}}_{\widehat{X}}\xx\widehat{\mathcal{G}}_o\in C^{\infty}_{c}(\R^3)\,,
$$
for which there exist two positive constants $C_1,C_2$ such that 
$$
C_1 \vert x\vert^{-1} \leqslant \widehat{\mathcal{G}}_o\leqslant C_2 \vert x\vert^{-1}
$$
outside a sufficiently large open ball.\\
By applying Theorem~A.33 in~\cite{DanLee}, we then get $ \widehat{\mathcal{G}}_o=O_2(\vert x\vert^{-1})$ and by combining this result with the fact that $\widehat{\mathcal{L}}_{\widehat{X}} \xx \widehat{\mathcal{G}}_o\in C^{\infty}_{c}(\R^3)$ and $\widehat{g}_{ij}-\delta_{ij}=O_2(\vert x\vert^{-\tau})$, it follows that 
$$
\Delta_{e} \widehat{\mathcal{G}}_o=f=O(\vert x\vert^{-3-\tau})\,.
$$
Assuming, without loss of generality, that $\tau\in (1/2,1)$ and by exploiting the asymptotic behavior of the function $f$, we then have
$$
\widehat{\mathcal{G}}_o(x)=-\frac{1}{4\pi}\int\limits_{\R^{3}}\frac{f(y)}{\vert x-y \vert}\,dy\,.
$$
Let us now rewrite $\widehat{\mathcal{G}}_o$ in $\R^{3}\setminus \{0\}$ in the following way: 
\begin{align*}
\widehat{\mathcal{G}}_o(x)&=-\frac{1}{4\pi\vert x\vert}\int\limits_{\R^{3}}f(y) \,dy+\frac{1}{4\pi\vert x\vert}\int\limits_{\{\vert y\vert\, \geqslant\,\vert x \vert/2\,\}}\!\!\!\!f(y)\, dy+\frac{1}{4\pi}\!\!\int\limits_{\{\vert y\vert\,<\,\vert x \vert/2\,\}}\!\!\!\Big(\frac{1}{\vert x\vert}-\frac{1}{\vert x-y \vert}\Big)\,f(y)\,dy\\
&\quad-\frac{1}{4\pi}\!\!\int\limits_{\{\vert y-x\vert\,<\,\vert x \vert/2\,\}}\!\!\!\frac{f(y)}{\vert x-y \vert}\,dy-\frac{1}{4\pi}\!\!\int\limits_{\{\vert y-x\vert\, \geqslant\,\vert x \vert/2\,\text{and}\,\vert y\vert\, \geqslant\,\vert x \vert/2 \}}\!\!\!\frac{f(y)}{\vert x-y \vert}\,dy\,.
\end{align*}
Since $f=O(\vert x\vert^{-3-\tau})$ and $\{\vert y-x\vert\,<\,\vert x \vert/2\,\}\cap\{\vert y\vert\,<\,\vert x \vert/2\,\}=\emptyset$, the absolute values of the second, fourth and fifth integrals can be easily bounded by $C\vert x\vert^{-1-\tau}$, for some constant $C>0$. 
Concerning the third integral, we start by observing that
\begin{align*}
\Bigg\vert \frac{1}{\vert x\vert}-\frac{1}{\vert x-y \vert}\Bigg\vert=&\,\Bigg\vert \frac{\vert x-y \vert^2-\vert x\vert^2}{\vert x\vert \,\vert x-y \vert\,\big( \vert x-y \vert+\vert x\vert\big)}\Bigg\vert= \Bigg\vert \frac{\vert y \vert^2-2 x^i y^i}{\vert x\vert \,\vert x-y \vert\,\big( \vert x-y \vert+\vert x\vert\big)}\Bigg\vert\\
\leqslant&\, \frac{\vert y \vert^2}{\vert x\vert^2 \,\vert x-y \vert} +\frac{2 \vert y\vert}{\vert x\vert\,\vert x-y \vert} \leqslant \frac{2\vert y \vert^2}{\vert x\vert^3} +\frac{4 \vert y\vert}{\vert x\vert^2}\,,
\end{align*}
where the last inequality follows from $\{\vert y-x\vert\,<\,\vert x \vert/2\,\}\cap\{\vert y\vert\,<\,\vert x \vert/2\,\}=\emptyset$. Hence, the decay rate $f=O(\vert x\vert^{-3-\tau})$, with $\tau\in (1/2,1)$, implies that the absolute value of the third integral can also be bounded by $C\vert x\vert^{-1-\tau}$, for some constant $C>0$. 
Thus, we have 
\begin{equation*}
\widehat{\mathcal{G}}_o(x)=-\frac{1}{4\pi\vert x\vert}\int\limits_{\R^{3}}f(y) \,dy+z(x)
\end{equation*}
in $\R^{3}\setminus \{0\}$, with $z=O(\vert x\vert^{-1-\tau})$.\\
We notice that the function $\psi_{R}z$ satisfies $\widehat{\mathcal{L}}_{\widehat{X}}\,(\psi_{R}z)=O_{1}(\vert x\vert^{-3-\tau})$, hence, applying again Theorem~A.33 in~\cite{DanLee}, we obtain $z=O_{2}(\vert x\vert^{-1-\tau})$. Finally, combining these results with formula~\eqref{eq33}, we conclude 
\begin{equation}
\mathcal{G}_o=\frac{\ {A}\ }{\vert x\vert}+O_{2}(\vert x\vert^{-1-\tau})\,,
\end{equation}
where ${A}$ is a positive constant.

\section{A monotonicity formula}\label{monotonicity formula section}

Before showing the proof of Theorem~\ref{GenPMTmon}, we recall some results that guarantee that the function $F$ is well--defined.

\begin{itemize}
\item Every level set of $u\,=\,1-4\pi \mathcal{G}_o$ is compact and has finite
$2$--dimensional Hausdorff measure $\mathcal{H}^2$ of $(M,g)$, by~\cite[Theorem~1.7]{Hardt1}.
\item The set $\left\lbrace\lvert\nabla u\rvert=0\right\rbrace$ has locally finite $1$--dimensional Hausdorff measure, see~\cite[Theorem~1.1]{Hardt2}. In particular, $\mathcal{H}^2(\{\vert \nabla u\vert=0\})=0$.
\item If $f$ is a continuous function on $M\setminus\{o\}$, the equality
$$\int\limits_{\{u=s\}}f\,d\mathcal{H}^2=\int\limits_{\{u=s\}^*}\!f\,d\sigma$$
holds, as a consequence of two previous points, where $\sigma$ denotes the canonical (volume) measure associated with the induced Riemannian metric on $\{u=s\}^*\!=\!\{u=s\}\!\setminus\{\vert \nabla u\vert=0\}$.
\end{itemize}

\begin{proof}[Proof of Theorem~\ref{GenPMTmon}]
As before, we proceed by steps.

\medskip

\noindent\textbf{Step~1.} {\em The auxiliary function
\begin{equation}\label{wideF}
\widetilde{F}\,:\,s\in (-\infty,1)\longmapsto \int\limits_{\{u=s\}}\!\!\!\vert \nabla u\vert^{-1}g(\nabla \vert \nabla u\vert, \nabla u)\, d\mathcal{H}^2\in\R
\end{equation}
belongs to $W^{1,1}_{\mathrm{loc}}(-\infty,1)$.}\\
To show this, we start by observing that $\widetilde{F}\in L^1_{\mathrm{loc}}(-\infty,1)$, as a consequence of the coarea formula~\cite[Proposition 2.1 and Remark 2.2]{BFM}. Let us consider the vector field $Y$, given by
\begin{align}\label{Y}
Y=\nabla \vert \nabla u\vert\,,
\end{align}
which is well--defined and smooth on the open set $M_{o}\setminus \mathrm{Crit}(u)$, where $M_{o}$ is defined as
$$
M_{o}=M\setminus\{o\} \, .
$$ 
Notice that 
\begin{equation*}
\widetilde{F}(s)\,=\int\limits_{\{u=s\}}\! \!\!g\bigg(Y , \, \frac{\nabla u}{|\nabla u|}\bigg) \,d\mathcal{H}^2
\end{equation*}
everywhere and the divergence of $Y$ on $M_{o}\setminus \mathrm{Crit}(u)$ can be expressed as 
\begin{equation}\label{divY}
\mathrm{div}(Y)=\,\vert \nabla u \vert^{-1}\Big(
\vert \nabla du\vert^{2}\,-\,\vert\,\nabla\vert \nabla u\vert\,\vert^{2}\,+\,\Ric(\na u,\na u)\,+\,g(\nabla \Delta u,\nabla u)\Big)\,,
\end{equation}
by the Bochner formula. Moreover, we also have
\begin{equation*}
2 g(\nabla \Delta u,\nabla u)\,=\,2 \nabla u\,(\Delta u)\,=\,\nabla u\,(g(X,\nabla u))\,=\,g\big(\nabla_{\!\nabla u}X, \nabla u\big)\,+\,\nabla du(X, \nabla u)\,,
\end{equation*}
since $u$ satisfies the equation $\mathcal{L}_Xu=0$ in $M_{o}$.
We claim that 
\begin{equation}
\label{eq:claim}
\mathrm{div}(Y)\in L^{1}_{\mathrm{loc}}(M_o) \,.
\end{equation}
Indeed, if $K$ is a compact subset of $M_o$, by Sard's Theorem, $K\subseteq E_{s}^{S}=\left\{s<u<S\right\}$, for some regular values $s,S$ of $u$ such that $-\infty<s<S<1$. Then, similarly to the proof of Theorem~1.1 in~\cite{Ago_Maz_Oro_2}, let us consider a sequence of smooth nondecreasing cut--off functions $\{ \eta_k\}_{k \in \mathbb{N}}$ on $(0,+\infty)$ such that
\begin{align*}
\eta_k \equiv 0\ \ \text{in $\left(0 \,,\frac{1}{3^k}\,\right]$}\,,\qquad 
0 \leqslant \eta_k \leqslant 3^k\ \ \text{in $\left[\,\frac{1}{3^k}\,,\frac{1}{3^{k-1}}\,\right]$}\,,
\qquad\eta_k \equiv 1\ \ \text{in $\left[\,\frac{1}{3^{k-1}} \, ,+\infty\!\right)$}\,.
\end{align*}
noticing that this sequence is converging pointwise monotonically to the function identically equal to $1$ on $(0,+\infty)$. We use these cut--off functions to define, for every $k \in\mathbb{N}$, the vector field
\begin{equation*}
Y_{k}\, = \, \, \eta_{{k}}\big( \vert \nabla u\vert\big)\,Y \,,
\end{equation*}
which is smooth on $M_{o}$. 
For any such $Y_k$, the divergence is given by
\begin{align}\label{eq2}
\mathrm{div}(Y_{k})
&=\, \eta_{{k}}\big( \vert \nabla u\vert\big)Y \,+\,\eta_{{k}}'\big( \vert \nabla u\vert\big)\,\vert\,\nabla\vert \nabla u\vert\,\vert^{2}\,,
\end{align}
and, on any compact subset of $M_{o}\setminus \mathrm{Crit}(u)$, it coincides with the vector field $Y$, provided $k$ is large enough.
By these considerations and applying the divergence theorem, there holds
\begin{equation}
\label{tfci}
\widetilde{F}(S)-\widetilde{F}(s) \,=\, \int\limits_{E_s^S} \!\mathrm{div}(Y_k) \, d\mu \, \geqslant \, \int\limits_{E_s^S} \! P_k \, d\mu \,+ \,\int\limits_{E_s^S}\!D_k \, d\mu\,,
\end{equation}
where we set 
$$
P_k \, = \, \eta_{{k}}\big(\vert \nabla u\vert\big)\,P\qquad\text{ and }\qquad 
D_k \, =\, \eta_{{k}}\big(\vert \nabla u\vert \big)\,D\,,
$$
with
$$
P\,=\,\vert \nabla u \vert^{-1}\Big[ \vert \nabla du\vert^{2}\,-\,\vert\,\nabla\vert \nabla u\vert\,\vert^{2}\Big] \mathds{1}_{M_{o}\setminus \mathrm{Crit}(u)}\, \geqslant \,0\,,$$
and
$$
D\,=\,\vert \nabla u \vert^{-1}\,\Big[\Ric(\na u,\na u)\,+\,\frac{1}{2}\big(\nabla X\big)^{\!\flat}(\nabla u,\nabla u)\,+\,\frac{1}{2}\nabla du(X, \nabla u) \Big]\mathds{1}_{M_{o}\setminus \mathrm{Crit}(u)}\, .
$$
Now, since the functions $D_k$ satisfy the inequality
$$
\vert D_k \vert
\, \leqslant \,\vert \nabla u \vert \, (\vert\Ric\vert \,+\,\vert \nabla X\vert)\,+\,\vert X\vert\, \vert \nabla du\vert\,\in L^{1}_{\mathrm{loc}}(M_o) \,,
$$
applying the dominated convergence theorem, we have $D\in L^{1}(E_{s}^{S})$ and 
\begin{equation*}
\lim_{k\to + \infty} \ \int\limits_{E_s^S}\! D_k \, d\mu\,= \,\int\limits_{E_s^S} \!D \, d \mu\, < +\infty \,.
\end{equation*}
This fact, combined with inequality~\eqref{tfci}, implies that the sequence of the integrals of the functions $P_k$ is uniformly bounded from above. We now notice that, since $P\geqslant 0$ and the sequence $\{ \eta_k\}_{k \in \mathbb{N}}$ is nondecreasing, the nonnegative functions $P_k$ converge monotonically pointwise to the function $P$ on $M_o$. Thus, the monotone convergence theorem yields 
\begin{equation*}
\lim_{k\to + \infty} \ \int\limits_{E_s^S}\!P_k \, d\mu\, = \int\limits_{E_s^S} \!P\, d \mu \, < +\infty\,.
\end{equation*}
In particular, we have $P\in L^{1}(E_{s}^{S})$.
Since $\mathrm{div}(Y)\mathds{1}_{M_{o}\setminus \mathrm{Crit}(u)}=P+D$, it then follows that $\mathrm{div}(Y)\in L^{1}_{\mathrm{loc}}(M_o)$, as claimed. \\
We are now ready to prove that $\widetilde{F} \in W^{1,1}_{\mathrm{loc}}(-\infty,1)$ with weak derivative given by
\begin{align}\label{Phi'}
\widetilde{F}'(s)&\,=\int\limits_{\{u=s\}}\!\!\! \vert \nabla u \vert^{-1}\, \mathrm{div}(Y)\,d\mathcal{H}^2
\end{align}
almost everywhere in $(-\infty,1)$. First we observe that the right--hand side is a function belonging to $L^1_{\mathrm{loc}}(-\infty,1)$, by applying the coarea formula~\cite[Proposition 2.1 and Remark 2.2]{BFM}, coupled with the result $\mathrm{div}(Y)\in L^{1}_{\mathrm{loc}}(M_o)$. Then, considering a test function $\psi\in C_{c}^{\infty}(-\infty,1)$, we have 
\begin{align*}
\int\limits_{-\infty}^{1}\!\!\psi'(\tau)\,\widetilde{F}(\tau)\,d\tau&\,=\,\int\limits_{-\infty}^{1}d\tau\int\limits_{\{u=s\}}\!\!\!\psi'(u)\,|\nabla u|^{-1}g(Y , \nabla u) \,d\mathcal{H}^2\,=\,\int\limits_{M_{o}}\!g\big(Y,\,\nabla \psi (u)\big)\,d\mu\\
&\,=\,\lim_{k\to +\infty}\,\int\limits_{M_{o}}\!g\big(Y_k,\,\nabla \psi (u)\big)\,d\mu\,=\,-\,\lim_{k\to +\infty}\,\int\limits_{M_{o}}\!\psi(u)\,\mathrm{div}( Y_{k})\,d\mu\,,
\end{align*}
where the second equality follows by the coarea formula, the third one by the dominated convergence theorem, whereas the last one is a simple integration by parts. Let us put $-\infty<s<S<1$ such that $\mathrm{supp}\,\psi \subseteq (s,S)$ and $s,S$ are regular values of $u$. 
By identity~\eqref{eq2}, it follows that
\begin{align*}
\int\limits_{M_{o}}\!\psi(u)\,\mathrm{div}( Y_{k})\,d\mu&\,=\int\limits_{E_{s}^{S}}\!\psi(u)\,\Big[P_{k}\,+\,D_{k}\,+\,\eta_{{k}}'\big( \vert \nabla u\vert\big)\,\vert\,\nabla\vert \nabla u\vert\,\vert^{2} \Big]\,d\mu\,.
\end{align*}
The standard identity 
\begin{equation}\label{MarFarVal}
\vert \nabla du\vert^{2}-\vert\,\nabla\vert \nabla u\vert\,\vert^{2}=\vert \nabla u\vert^{2} \vert \mathrm{h}\vert^{2}+\vert\,\nabla^{\top}\vert \nabla u\vert\,\vert^{2}=\vert \nabla u\vert^{2} \vert \ringg{\mathrm{h}}\vert^{2}+\vert\,\nabla^{\top}\vert \nabla u\vert\,\vert^{2}+\vert \nabla u\vert^{2} (\mathrm{H}^2/2)\,,
\end{equation}
along with the fact that 
\begin{equation}\label{exprH}
\mathrm{H}\,=\, \frac{\Delta u}{\vert \nabla u\vert}\,-\,\frac{g(\nabla \vert \nabla u\vert,\nabla u) }{\vert \nabla u\vert^2}\,=\, \frac{1}{2}\, \frac{g(X ,\nabla u)}{|\nabla u|}\,-\,\frac{g(\nabla \vert \nabla u\vert,\nabla u) }{\vert \nabla u\vert^2}\,,
\end{equation} 
gives
\begin{align*}
\vert \nabla du\vert^{2}-\vert\,\nabla\vert \nabla u\vert\,\vert^{2}&\, \geqslant \, \frac{1}{2}\,\vert\,\nabla^{\top}\vert \nabla u\vert\,\vert^{2}\,+\,\frac{1}{2}\,\vert\,\nabla^{\perp}\vert \nabla u\vert\,\vert^{2} \,-\,\frac{1}{2}\,\vert X\vert \,\vert \nabla u\vert\, \vert\,\nabla\vert \nabla u\vert\,\vert\\
&\, \geqslant \,\frac{1}{2}\,\vert\,\nabla\vert \nabla u\vert\,\vert^{2} \,-\,\frac{1}{2}\,\vert X\vert \,\vert \nabla u\vert\, \vert\,\nabla du\vert\,
\end{align*}
which leads to
$$
\vert \nabla u \vert^{-1}\,\vert\,\nabla\vert \nabla u\vert\,\vert^{2}\, \leqslant \,2 P\,+\,\vert X\vert \, \vert\,\nabla du\vert \,,
$$
outside the set of the critical points of $u$, so $\vert \nabla u \vert^{-1}\,\vert\,\nabla\vert \nabla u\vert\,\vert^{2}\in L^1_{\mathrm{loc}}(M_o)$, whereas $|\na u|\, \eta'_{{k}}\big(\vert \nabla u\vert\big)$ is always bounded. Then, as $\lim_{k \to + \infty}\eta_k'(\tau)=0$ for every $\tau \in (0, + \infty)$, the dominated convergence theorem implies
$$
\lim_{k\to +\infty}\,\int\limits_{E_{s}^{S}}\!\psi(u)\,\eta_{{k}}'\big( \vert \nabla u\vert\big)\,\vert\,\nabla\vert \nabla u\vert\,\vert^{2} \,d\mu\,=\,0\,.
$$
In conclusion, we obtain
\begin{align*}
\int\limits_{-\infty}^{1}\!\psi'(\tau)\,\widetilde{F}(\tau)\,d\tau\,=&\,-\,\lim_{k\to +\infty}\int\limits_{M_{o}}\!\psi(u)\,\mathrm{div}( Y_{k})\,d\mu=\,-\,\int\limits_{M_{o}} \!\psi(u)\,\mathrm{div}( Y)\,d\mu\\
=&\,-\,\int\limits_{-\infty}^{1} \!\psi(\tau)\int\limits_{\{u=\tau\}}\! \!\vert \nabla u \vert^{-1}\, \mathrm{div}(Y)\,d\mathcal{H}^2 \,d\tau\,,
\end{align*}
where in the last identity we used again the coarea formula. The first step clearly follows. 

\medskip

\noindent\textbf{Step~2.} {\em The function
$$
\widetilde{F}_1\,:\,s\in (-\infty,1)\longmapsto \int\limits_{\{u=s\}}\!\!\!g\big(X, \nabla u\big) \, d\mathcal{H}^2\in\R
$$
belongs to $W^{1,1}_{\mathrm{loc}}(-\infty,1)$.}\\
By the coarea formula, $\widetilde{F}_1\in L^1_{\mathrm{loc}}(-\infty,1)$. We consider the vector field $Z$, given by
\begin{align}\label{Z}
Z\,=\,\vert \nabla u\vert \,X\,,
\end{align}
on the open set $M_{o}\setminus \mathrm{Crit}(u)$, which satisfies everywhere
\begin{equation}\label{wideF1}
\widetilde{F}_1(s)\,=\int\limits_{\{u=s\}}\! \!\!g\bigg(Z , \, \frac{\nabla u}{|\nabla u|}\bigg) \, d\mathcal{H}^2
\end{equation}
and 
\begin{equation}\label{divZ}
\mathrm{div}(Z)\,=\,\vert \nabla u \vert\, \mathrm{div}(X) \,+\, g(\nabla \vert \nabla u\vert, X)
\end{equation}
on $M_{o}\setminus \mathrm{Crit}(u)$.\\
In this case, $\mathrm{div}(Z)\in L^{1}_{\mathrm{loc}}(M_o)$ trivially holds, therefore, by the coarea formula~\cite[Proposition 2.1 and Remark 2.2]{BFM}, the function defined almost everywhere in $(-\infty,1)$,
$$
s\longmapsto \int_{\{u=s\}}\! \vert \nabla u \vert^{-1}\, \mathrm{div}(Z)\,d\mathcal{H}^2
$$
belongs to $L^1_{\mathrm{loc}}(-\infty,1)$. Considering a test function $\psi\in C_{c}^{\infty}(-\infty,1)$, we obtain
\begin{align*}
\int\limits_{-\infty}^{1}\!\!\psi'(\tau)\,\widetilde{F}_1(\tau)\,d\tau&\,=\,\int\limits_{-\infty}^{1}d\tau\int\limits_{\{u=s\}}\!\!\!\psi'(u)\,|\nabla u|^{-1} g(Z ,\nabla u)\,d\mathcal{H}^2\,=\,\int\limits_{M_{o}}\!g\big(Z,\,\nabla \psi (u)\big)\,d\mu\\
&\,=\,\lim_{k\to +\infty}\,\int\limits_{M_{o}}\!g\big(Z_k,\,\nabla \psi (u)\big)\,d\mu\,=\,-\,\lim_{k\to +\infty}\,\int\limits_{M_{o}}\!\psi(u)\,\mathrm{div}( Z_{k})\,d\mu\,,
\end{align*}
where $Z_k$ denotes the smooth vector field $\eta_{{k}}\big( \vert \nabla u\vert\big)\,Z$. Let $s, S\in (-\infty,1)$ be such that $\mathrm{supp}\,\psi \subseteq (s,S)$ and $s,S$ are regular values of $u$. Then, there holds
\begin{align*}
\int\limits_{M_{o}}\!\psi(u)\,\mathrm{div}( Z_{k})\,d\mu&\,=\int\limits_{E_{s}^{S}}\!\psi(u)\,\Big[\eta_{{k}}\big( \vert \nabla u\vert\big)\,\mathrm{div}( Z)\,+\,|\na u|\,\eta_{{k}}'\big( \vert \nabla u\vert\big)\,g(\nabla\vert \nabla u\vert, X) \Big]\,d\mu\,,
\end{align*}
which, by the dominated convergence theorem, leads to
\begin{align*}
\int\limits_{-\infty}^{1}\!\psi'(\tau)\,\widetilde{F}_1(\tau)\,d\tau\,=&\,-\,\lim_{k\to +\infty}\int\limits_{M_{o}}\!\psi(u)\,\mathrm{div}( Z_{k})\,d\mu\,=\,-\,\int\limits_{M_{o}} \!\psi(u)\,\mathrm{div}( Z)\,d\mu\\
\,=&\,-\,\int\limits_{-\infty}^{1} \!\psi(\tau)\int\limits_{\{u=\tau\}}\! \!\vert \nabla u \vert^{-1}\, \mathrm{div}(Z)\,d\mathcal{H}^2 \,d\tau\,.
\end{align*}
In particular, 
\begin{align}\label{Fwide1'}
\widetilde{F}'_1(s)&\,=\int\limits_{\{u=s\}}\!\!\! \vert \nabla u \vert^{-1}\, \mathrm{div}(Z)\,d\mathcal{H}^2
\end{align}
almost everywhere in $(-\infty,1)$.

\medskip

\noindent\textbf{Step~3.} {\em The function 
$$
\widetilde{F}_2\,:\,s\in (-\infty,1)\longmapsto \int\limits_{\{u=s\}}\!\!\!\vert \nabla u\vert^2 \ d\mathcal{H}^2\in\R
$$
belongs to $W^{2,1}_{\mathrm{loc}}(-\infty,1)$.}\\
By arguing as before, but replacing the vector field $Z$ with a new vector field
\begin{equation}\label{W}
W\,=\,\vert \nabla u\vert \,\nabla u\,,
\end{equation}
we have that $\widetilde{F}_2\in W^{1,1}_{\mathrm{loc}}(-\infty,1)$, with weak derivative given by
\begin{align}\label{Fwide2'}
\widetilde{F}'_2(s)&\,=\,\int\limits_{\{u=s\}}\!\!\! \vert \nabla u \vert^{-1}\, \mathrm{div}(W)\,d\mathcal{H}^2\,=\,\frac{1}{2}\int\limits_{\{u=s\}}\!\!\! g(X,\nabla u)\,d\mathcal{H}^2\,+\,\int\limits_{\{u=s\}}\!\!\!\vert \nabla u\vert^{-1}g(\nabla \vert \nabla u\vert, \nabla u)\,d\mathcal{H}^2
\end{align}
almost everywhere in $(-\infty,1)$. Since the right--hand side of the this equality belongs to the space $W^{1,1}_{\mathrm{loc}}(-\infty,1)$ by the previous two steps, the conclusion of this step follows.

\medskip

\noindent\textbf{Step~4.} {\em Every regular level set of $u$ is either connected, or none of its connected components is a $2$--sphere.}\\
We start by observing that if we set $u(o)=-\infty$, the function $u:M\to [-\infty,1)$ is continuous and proper.
Then, for any regular value $s\in (-\infty,1)$ of $u$, the set $M\setminus \{u=s\}$ is a disjoint union of connected open sets. Among these, there exists a unique bounded open set that coincides with $\{u<s\}$, whereas each of the remaining components is unbounded and their union is $\{u>s\}$. Indeed, since $u:M\to [-\infty,1)$ is continuous and proper, the set $\{u<s\}$ is bounded and it must be the (finite) union of disjoint bounded and connected open sets each one with boundary given by some union of connected components of the regular level set $\{u=s\}$. Then, if $o$ does not belong to one of such open sets, by the maximum principle ($u$ is bounded there) it would follow that $u$ is constant, which is a contradiction with the regularity of the level set. Hence, the set $\{u<s\}$ is a single bounded and connected open set containing the pole $o$. Arguing similarly, one can prove that also all connected components of $\{u>s\}$ are unbounded. 
Now, we assume that the regular level set $\{u=s\}$ has a connected component $\Sigma$ that is a $2$--sphere. Since the second integral homology group $H_2(M;\Z)$ does not contain any ``spherical'' class, the surface $\Sigma$ is the boundary of a bounded and connected open set $\Omega$ (see e.g.~\cite{Hatcher}). If $o$ does not belong to $\Omega$, then $\overline{\Omega}$ is contained in $M\setminus\{o\}$ with boundary $\Sigma$ and this is impossible, again by the maximum principle, as above. Therefore, the pole $o$ must belong to $\Omega$, hence the function $u$ tends to $-\infty$ getting close to the pole and $u=s$ on $\partial \Omega$. The maximum principle then implies that $u<s$ everywhere in $\Omega$, thus, $\Omega \subseteq \{u<s\}$.\\
If there is a point $p\in M$ which belongs to $\{u<s\}$ and not to $\Omega$, as the set $\{u<s\}$ is connected, there is a path from $o$ to $p$ contained in $\{u<s\}$. Due to the fact that $o\in \Omega$, this path must intersect its boundary where the value of the function $u$ is $s$. Since this is impossible, we conclude that $\Omega= \{u<s\}$, hence $\Sigma=\partial\Omega=\partial\{u<s\}=\{u=s\}$, implying that this level set is connected.

\medskip

\noindent\textbf{Step~5.} {\em Conclusion of the proof.}\\
Putting together the information collected in the first three steps, we get $F\in W^{1,1}_{\mathrm{loc}}(0,+\infty)$, in particular, it admits a locally absolutely continuous representative in $(0,+\infty)$ that coincides with $F$ defined pointwise in the statement of Theorem~\ref{GenPMTmon} on the set
\begin{equation}\label{mathcalT}
\mathcal{T}\,=\,\{t\in (0,+\infty)\,:\, 1-1/t\, \,\text{is a regular value of $u$}\,\}\,.
\end{equation}
We know that, by the completeness of $M$, the set $\mathcal{T}$ is open in $(0,+\infty)$ and that, by the Sard's theorem, the complement of $\mathcal{T}$ in $(0,+\infty)$ has zero Lebesgue measure. Notice that the function $F$ is at least continuously differentiable on the open set $\mathcal{T}$ and that, by previous results, we can write
\begin{align*}
F(t)&\,-\,F(s)\,=\!\!\int\limits_{(s,t)\cap \mathcal{T}}\!\!F'(\tau)\,d\tau\\
&\,=\!\!\int\limits_{(s,t)\cap \mathcal{T}}\!\!\Bigg\{4\pi\,+\, 3\tau^2\!\int\limits_{\Sigma_\tau}\vert \nabla u \vert^{2} \, d\mathcal{H}^2\,+\,\tau \!\int\limits_{\Sigma_\tau}\bigg[\,\frac{1}{2}g(X,\nabla u)\,+\,\vert \nabla u\vert^{-1}g(\nabla \vert \nabla u\vert, \nabla u)\bigg]\,d\mathcal{H}^2\\
&\phantom{\,=\!\!\int\limits_{(s,t)\cap \mathcal{T}}\!\!\Bigg\{}\,-\ 2\tau \!\int\limits_{\Sigma_\tau}\bigg[\,\frac{1}{2}\, g(X , \nabla u)\,-\,\vert \nabla u\vert^{-1}g(\nabla \vert \nabla u\vert, \nabla u) \bigg]\, d\mathcal{H}^2\\
&\phantom{\,=\!\!\int\limits_{(s,t)\cap \mathcal{T}}\!\!\Bigg\{}\,-\,\frac{1}{2}\,\int\limits_{\Sigma_\tau}\bigg[\,\mathrm{div}(X) \,+\,\vert \nabla u\vert^{-1}g(\nabla \vert \nabla u\vert,X) \bigg]\, d\mathcal{H}^2 \\
&\phantom{\,=\!\!\int\limits_{(s,t)\cap \mathcal{T}}\!\!\Bigg\{}\,+\,\int\limits_{\Sigma_\tau}\vert \nabla u \vert^{-2}\bigg[\,
\vert \nabla du\vert^{2}\,-\,\vert\,\nabla\vert \nabla u\vert\,\vert^{2}\,+\,\Ric(\na u,\na u)\,+\,\frac{1}{2}\,g\big(\nabla_{\!\nabla u}X, \nabla u\big)\\
&\phantom{\,=\!\!\int\limits_{(s,t)\cap \mathcal{T}}\!\!\Bigg\{}\qquad\qquad\quad\qquad\,+\,\frac{1}{2}\,\vert \nabla u\vert\, g(\nabla\vert \nabla u\vert, X)\bigg]\, d\mathcal{H}^2 \\
&\phantom{\,=\!\!\int\limits_{(s,t)\cap \mathcal{T}}\!\!\Bigg\{}\,+2\tau \,\int\limits_{\Sigma_\tau} g(X,\nabla u) \, d\mathcal{H}^2\,+\ \int\limits_{\Sigma_\tau}\bigg[\,\mathrm{div}(X) \,+\,\vert \nabla u\vert^{-1}g(\nabla \vert \nabla u\vert,X) \bigg]\, d\mathcal{H}^2\Bigg\}\,d\tau\,,
\end{align*}
where $\Sigma_\tau$ denotes the level set $\{u=1-1/\tau\}$ and we used the identity 
$$
\nabla du(X, \nabla u)=\vert \nabla u\vert g(X, \nabla\vert \nabla u\vert)\,.
$$
Now, by using the identities~\eqref{MarFarVal}-\eqref{exprH} and the traced Gauss equation, we get
\begin{align*}
F(t)\,-\,F(s)&\,=\!\!\int\limits_{(s,t)\cap \mathcal{T}}\!\!\Bigg\{4\pi\ -\,\int\limits_{\Sigma_\tau}\frac{\mathrm{R}^{\Sigma_\tau}}{2}\,d\mathcal{H}^2 \,+\, 3\tau^2\!\int\limits_{\Sigma_\tau}\vert \nabla u \vert^{2} \, d\mathcal{H}^2\,+\,3\tau \!\int\limits_{\Sigma_\tau}\Big[\,g(X,\nabla u)\,-\,\vert \nabla u\vert\,\mathrm{H}\Big]\,d\mathcal{H}^2\\
&\phantom{\quad\ \ \,=\!\!\int\limits_{(s,t)\cap \mathcal{T}}\!\!\Bigg\{}\,+\,\int\limits_{\Sigma_\tau}\bigg[\frac{\mathrm{R}}{2}\,+\,\frac{\vert\,\nabla^{\top}\vert \nabla u\vert\,\vert^2}{\vert \nabla u\vert^{2}}\,+\,\frac{\vert \ringg{\mathrm{h}}\vert^{2}}{2}\,+\,\frac{3\mathrm{H}^2}{4}\bigg]\, d\mathcal{H}^2\\
&\phantom{\quad\ \ \,=\!\!\int\limits_{(s,t)\cap \mathcal{T}}\!\!\Bigg\{}\,-\,\frac{1}{2}\,\int\limits_{\Sigma_\tau}\bigg[\,\mathrm{div}(X)\,-\,\frac{g\big(\nabla_{\!\nabla u}X, \nabla u\big)}{\vert \nabla u \vert^{2}}\bigg]\, d\mathcal{H}^2\\
&\phantom{\quad\ \ \,=\!\!\int\limits_{(s,t)\cap \mathcal{T}}\!\!\Bigg\{}\,+\ \int\limits_{\Sigma_\tau}\bigg[\,\mathrm{div}(X)\,+\,\frac{g(\nabla \vert \nabla u\vert,X)}{\vert \nabla u\vert} \bigg]\, d\mathcal{H}^2\Bigg\}\,d\tau\,.
\end{align*}
Then, the identities
\begin{align*}
\mathrm{div}(X)\,-\,\vert \nabla u \vert^{-2}g\big(\nabla_{\!\nabla u}X, \nabla u\big)&\,=\,\mathrm{div}(X)\,-\,g\big(\nabla_{\nu}X, \nu\big)\,=\,\mathrm{div}^{\!\top}(X)=\,\mathrm{div}^{\!\top}(X^{\top})\,+\,g(X,\nu)\mathrm{H}\,,\\[0.2em]
g\bigg(\frac{\nabla \vert \nabla u\vert}{\vert \nabla u\vert},X\bigg)&\,=\,g\bigg(\frac{\nabla\vert \nabla u\vert}{\vert \nabla u\vert}, X^{\top}\bigg)+\,g(X,\nu)\,\frac{g( \nabla \vert \nabla u\vert, \nabla u)}{\vert \nabla u\vert^2}\\[0.2em]
&\,=\,g\bigg(\frac{\nabla^{\top}\vert \nabla u\vert}{\vert \nabla u\vert}, X^{\top}\bigg)+\,\frac{1}{2}\,\vert X^{\perp}\vert^2\,-\,g(X,\nu)\,\mathrm{H}\,,
\end{align*}
where $\nu=\nabla u/\vert \nabla u \vert$ and $\mathrm{div}^{\!\top}$ is the tangential divergence along a generic surface, imply
\begin{align*}
F(t)\,-\,F(s)&\,=\!\!\!\!\!\int\limits_{(s,t)\cap \mathcal{T}}\!\!\Bigg\{4\pi\ -\,\int\limits_{\Sigma_\tau}\frac{\mathrm{R}^{\Sigma_\tau}}{2}\,d\mathcal{H}^2 \,+\, 3\tau^2\!\int\limits_{\Sigma_\tau}\vert \nabla u \vert^{2} \, d\mathcal{H}^2\,+\,3\tau \!\int\limits_{\Sigma_\tau}\Big[\,g(X,\nabla u)\,-\,\vert \nabla u\vert\,\mathrm{H}\Big]\,d\mathcal{H}^2\\
&\phantom{\ \ \!=\!\!\!\!\!\int\limits_{(s,t)\cap \mathcal{T}}\!\!\Bigg\{}-\frac{3}{2}\int\limits_{\Sigma_\tau} g(X,\nu)\mathrm{H}\,d\mathcal{H}^2+\int\limits_{\Sigma_\tau}\bigg[\frac{\mathrm{R}}{2}+\,\mathrm{div}(X)+\frac{\vert\nabla^{\top}\vert \nabla u\vert\vert^2}{\vert \nabla u\vert^{2}}+g\bigg(\frac{\nabla^{\top}\vert \nabla u\vert}{\vert \nabla u\vert}, X^{\top}\bigg)\\
&\phantom{\ \quad\qquad\qquad\qquad\ \,\quad\qquad\quad\ \ \ \ \ \!=\!\!\!\!\!\int\limits_{(s,t)\cap \mathcal{T}}\!\!\Bigg\{}\!+\,\frac{\vert X^{\perp}\vert^2}{2}\,+\,\frac{\vert \ringg{\mathrm{h}}\vert^{2}}{2}\,+\,\frac{3\mathrm{H}^2}{4}\bigg]\, d\mathcal{H}^2\Bigg\}\,d\tau\\
&=\!\!\!\!\!\int\limits_{(s,t)\cap \mathcal{T}}\!\!\Bigg\{4\pi\ -\,\int\limits_{\Sigma_\tau}\frac{\mathrm{R}^{\Sigma_\tau}}{2}\,d\mathcal{H}^2 \,+\, 3\tau^2\!\int\limits_{\Sigma_\tau}\vert \nabla u \vert^{2} \, d\mathcal{H}^2\,+\,3\tau \!\int\limits_{\Sigma_\tau}\Big[\,g(X,\nabla u)\,-\,\vert \nabla u\vert\,\mathrm{H}\Big]\,d\mathcal{H}^2\\
&\phantom{\ \ \!=\!\!\!\!\!\int\limits_{(s,t)\cap \mathcal{T}}\!\!\Bigg\{}-\frac{3}{2}\int\limits_{\Sigma_\tau} g(X,\nu)\mathrm{H}\,d\mathcal{H}^2+\int\limits_{\Sigma_\tau}\bigg[\frac{\mathrm{R}}{2}+\,\mathrm{div}(X)+\bigg\vert \frac{\nabla^{\top}\vert \nabla u\vert}{\vert \nabla u\vert}+\frac{1}{2}X^{\top}\bigg\vert^2\\
&\phantom{\qquad\qquad\ \,\quad\qquad\quad\qquad\quad\ \ \ \ \ \!=\!\!\!\!\!\int\limits_{(s,t)\cap \mathcal{T}}\!\!\Bigg\{}\,-\,\frac{\vert X^{\top}\vert^2}{4}\,+\,\frac{\vert X^{\perp}\vert^2}{2}\,+\,\frac{\vert \ringg{\mathrm{h}}\vert^{2}}{2}\,+\,\frac{3\mathrm{H}^2}{4}\bigg]\, d\mathcal{H}^2\Bigg\}\,d\tau\,.
\end{align*}
Therefore, by adding and subtracting the term $ \vert X^{\perp}\vert^2/4$, we obtain
\begin{align}
F(t)\,-\,F(s)
&=\!\!\!\!\!\int\limits_{(s,t)\cap \mathcal{T}}\!\!\Bigg\{4\pi\ -\,\int\limits_{\Sigma_\tau}\frac{\mathrm{R}^{\Sigma_\tau}}{2}\,d\mathcal{H}^2 \,+\, 3\tau^2\!\int\limits_{\Sigma_\tau}\vert \nabla u \vert^{2} \, d\mathcal{H}^2\,+\,3\tau \!\int\limits_{\Sigma_\tau}\Big[\,g(X,\nabla u)\,-\,\vert \nabla u\vert\,\mathrm{H}\Big]\,d\mathcal{H}^2\\
&\phantom{\ \ \ \!=\!\!\!\!\!\int\limits_{(s,t)\cap \mathcal{T}}\!\!\Bigg\{}+\,\int\limits_{\Sigma_\tau}\,-\,\frac{3}{2}g(X,\nu)\mathrm{H}+\bigg[\frac{\mathrm{R}}{2}\,+\,\mathrm{div}(X)\,-\,\frac{1}{4}\,\vert X\vert^2\,+\,\bigg\vert \frac{\nabla^{\top}\vert \nabla u\vert}{\vert \nabla u\vert}\,+\,\frac{1}{2}\,X^{\top}\bigg\vert^2\\
&\phantom{\ \ \ \,\!=\!\!\!\!\!\int\limits_{(s,t)\cap \mathcal{T}}\!\!\Bigg\{+\,\int\limits_{\Sigma_\tau}\,-\,\frac{3}{2}g(X,\nu)\mathrm{H}+\bigg[}+\,\frac{\vert \ringg{\mathrm{h}}\vert^{2}}{2}\,+\,\frac{3}{4}\,\vert X^{\perp}\vert^2\,+\,\frac{3\mathrm{H}^2}{4}\bigg]\, d\mathcal{H}^2\Bigg\}\,d\tau\\
&=\!\!\!\!\!\int\limits_{(s,t)\cap \mathcal{T}}\!\!\Bigg\{4\pi\ -\,\int\limits_{\Sigma_\tau}\frac{\mathrm{R}^{\Sigma_\tau}}{2}\,d\mathcal{H}^2 \,+\, \frac{3}{4}\int\limits_{\Sigma_\tau}\bigg[\frac{2\vert \nabla u \vert}{1-u}\,-\,\mathrm{H}\,+\,g(X,\nu)\bigg]^2\,d\mathcal{H}^2 \\
&\phantom{\ \ \!=\!\!\!\!\!\int\limits_{(s,t)\cap \mathcal{T}}\!\!\Bigg\{}+\int\limits_{\Sigma_\tau}\bigg[\frac{1}{2}\bigg(\mathrm{R}+2\,\mathrm{div}(X)- \frac{1}{2}\vert X\vert^2 \bigg)+\bigg\vert \frac{\nabla^{\top}\vert \nabla u\vert}{\vert \nabla u\vert}+\frac{1}{2}X^{\top}\bigg\vert^2+\frac{\vert \ringg{\mathrm{h}}\vert^{2}}{2}\bigg]\, d\mathcal{H}^2\Bigg\}\,d\tau\,.\label{feq1}
\end{align}
We now observe that the Euler characteristic of every regular level sets of $u$ is less than or equal to $2$. This follows from Step~4, taking into account that, since $M$ is orientable, every connected components of a regular level set of $u$ is itself orientable.
Consequently, the Gauss--Bonnet theorem implies 
$$
4\pi\ -\,\int\limits_{\Sigma_\tau}\frac{\mathrm{R}^{\Sigma_\tau}}{2}\,d\mathcal{H}^2\, \geqslant \,0\,,
$$
for every $\tau\in \mathcal{T}$. Then, the conclusion of Theorem~\ref{GenPMTmon} follows from equality~\eqref{feq1}, as 
\begin{equation}\label{feq2}
\mathrm{R}\,+\,2\,\mathrm{div}(X)\,-\, \frac{1}{2}\,\vert X\vert^2\,=\, \mathrm{R}_X^{(k)}\,+\,\frac{k+2}{2k}\,\vert X\vert^2 
\end{equation}
and $ \mathrm{R}_X^{(k)} \geqslant 0$, by assumption.
\end{proof}

\begin{remark}
We observe that by the same argument, under the assumptions $(a)$ and $(b)$ of Theorem~\ref{GenPMTmon}, if the function $u$ is any solution of $\mathcal{L}_X u=0$ taking values in an interval $(\tau_1,\tau_2)\subseteq (-\infty,1)$, the function $F$ defined by formula~\eqref{eq0} on the interval $I=(1/(1-\tau_1),1/(1-\tau_2))$ is nondecreasing on the set
\begin{equation}
\mathcal{T}\,=\,\{t\in I\,:\, 1-1/t\, \,\text{is a regular value of $u$}\,\}\,.
\end{equation}
\end{remark}

Combining Theorem~\ref{GenPMTmon} with standard facts about the asymptotic behavior of the Green's functions near a pole, one obtains the following rigidity statement.

\begin{proposition}\label{Constantcase}
Under the assumptions of Theorem~\ref{GenPMTmon}, if the function $F$ is constant on the set $\mathcal{T}$ (defined in formula~\eqref{mathcalT}), then:
\begin{itemize}
\item[$(a)$] If $\mathrm{R}_X^{(k)}=\mathrm{R}+2\,\mathrm{div}(X)-(1+1/k)\, \vert X\vert^2\, \geqslant\,0$, for some $k\in (-\infty,-2)\cup (0,+\infty)$, then $(M,g)$ is isometric to $(\R^3,g_{{\mathrm{eucl}}})$.
\item[$(b)$] If $\mathrm{R}_X^{(-2)}\!=\mathrm{R}+2\,\mathrm{div}(X)-(1/2)\, \vert X\vert^2\, \geqslant\,0$, then $(M,g)$ is conformally isometric to $(\R^3, g_{{\mathrm{eucl}}})$.
\end{itemize}
\end{proposition}

\begin{proof}
By Step~1 of Section~\ref{Greensection}, the operator $\mathcal{L}_X$ admits a unique positive minimal Green's function $\mathcal{G}^{B_{r_1}(o)}_o$ with pole at $o$ in an open ball $B_{r_1}(o)$, for any $r_1<\mathrm{inj}(o)$. It also turns out that, for any normal coordinate system $(x^1,x^2,x^3)$ centered at $o$ and defined on the ball $B_{r_1}(o)$, the asymptotic behaviors 
$$
\vert \mathcal{G}_{o} - 1/(4\pi r)\vert\, = \, o (r^{-1})\qquad\text{ and }\qquad\vert \nabla \mathcal{G}_{o}^{B_{r_1}(o)}\,+\,(\nabla r/ 4\pi r^{2})\vert\,= \, o(r^{-2})
$$
occur (here, $r$ denotes the distance function from the pole $o\in M$). Now, since $\mathcal{G}_{o}=\mathcal{G}_{o}^{B_{r_1}(o)} +w$, where $w$ is a smooth function that satisfies $\mathcal{L}_X w=0$ in $B_{r_1}(o)$ and an appropriate boundary condition, the function $\mathcal{G}_o$ has no critical points in a neighborhood of the point $o$.
Therefore, there exists a maximal time $T>0$ such that $\na u \neq 0$ in $u^{-1}(-\infty, 1 - 1/T)$. Then, the function $F$ is continuously differentiable on the interval $(0,T)$ and, by the assumption that $F$ is constant on the set $\mathcal{T}$, there holds $F'\equiv 0$ in $(0,T)$. Thus, all nonnegative summands in formula~\eqref{feq1} are forced to vanish, for every $t \in (0,T)$. Now, we need to distinguish between cases $(a)$ and $(b)$. In case $(a)$, by equality~\eqref{feq2}, the claim follows in the same way as in~\cite[Corollary 1.3]{Ago_Maz_Oro_2}. Notably, $X$ turns out to be zero on the whole $M$. 
In case $(b)$, the result that all nonnegative summands in formula~\eqref{feq1} are forced to vanish for every $t \in (0,T)$ implies that
\begin{equation}\label{feq3}
X=-\nabla \log\left(\frac{\vert \nabla u\vert^2}{(1-u)^4}\right)
\end{equation}
in the open set $u^{-1}(-\infty, 1 - 1/T)$. If $T<+\infty$, there exists $p\in \Sigma_{T}$ such that $\nabla u (p)=0$. We consider $\gamma:[0,\ell]\to M$ a smooth curve parametrized by the arclength from $o$ to $p$ and satisfiying $\gamma((0,\ell))\subseteq u^{-1}(-\infty, 1 - 1/T)$. Let $\widetilde{f}$ be the function given by 
$$\widetilde{f}(t)=-\log\left(\frac{\vert \nabla u\vert^2}{(1-u)^4}\right)(\gamma(t))$$
for every $t\in (0,\ell]$. We observe that
\begin{equation}
\widetilde{f}(\ell-\varepsilon)-\widetilde{f}(\ell/2)=\int\limits_{\ell/2}^{\ell-\varepsilon} \widetilde{f}'(t)\,dt=\int\limits_{\ell/2}^{\ell-\varepsilon} g\!\left(X(\gamma(t)),\gamma'(t) \right)dt\,,
\end{equation}
which yields
\begin{equation}
\vert \widetilde{f}(\ell-\varepsilon)\vert\, \leqslant \,\vert \widetilde{f}(\ell/2)\vert\,+\,\ell \,\sup_{M} \vert X\vert\,<\,\infty\,,
\end{equation}
for any $\varepsilon\in (0,\ell/2)$. This is not possible since $\widetilde{f}(t)\to+\infty$ as $t\to \ell^{-}$, by virtue of the fact that $\nabla u (p)=0$. Therefore, $T=+\infty$, $u$ has no critical points and $X$ is the gradient of the function $f$ given by
\begin{equation}
f=-\log\left(\frac{\vert \nabla u\vert^2}{(1-u)^4}\right)
\end{equation}
on $M\setminus\{o\}$. Notice that, since $X$ is a smooth vector field on the whole $M$ and $f(p)\to 0$ as $p\to o$, the equality~\eqref{feq3} also implies that $f$ is smooth on all $M$.
Since $|\nabla u| \neq 0$ everywhere, all the level sets of $u$ are regular and diffeomorphic to each other. More precisely, by the vanishing of the Gauss--Bonnet term in equality~\eqref{feq1}, they are all diffeomorphic to a $2$--sphere and $M$ is diffeomorphic to $\R^3$. Furthermore, the metric $g$ can be written on $M \setminus \{o\}$ as 
$$
g \, = \, \frac{du \otimes du}{\vert \nabla u\vert^2} \, +
 \, g_{\alpha \beta}
(u, \!\vartheta) \, d\vartheta^\alpha \otimes d\vartheta^\beta \, ,
$$ 
where $g_{\alpha\beta}(u, \!\vartheta) \, d\vartheta^\alpha \otimes d\vartheta^\beta$ represents the metric induced by $g$ on the level sets of $u$. Exploiting the vanishing of the traceless second fundamental form of the level sets in equality~\eqref{feq1}, together with 
$$
\mathrm{H}\,=\,\frac{2\vert \nabla u \vert}{1-u}\,+\,g(\nabla f,\nu)\,,
$$
it turns out that the coefficients $g_{\alpha\beta}(u,\!\vartheta)$ satisfy the following first order system of PDE's
$$
\partial_u g_{\alpha\beta}\,-\, g_{\alpha\beta} \partial_u f \, = \, 2 \, g_{\alpha\beta}/(1-u)\,,
$$
hence, 
$$
e^{-f(u, \vartheta)}g_{\alpha \beta}(u, \!\vartheta) \, d\vartheta^\alpha \otimes d\vartheta^\beta=(1-u)^{-2}c_{\alpha\beta}(\theta)\, d\vartheta^\alpha \otimes d\vartheta^\beta\,.
$$
Thus, we have
$$
\widetilde{g}=e^{-f}g\,=\,\frac{\vert \nabla u\vert^2}{(1-u)^4}\, \left(\frac{du \otimes du}{\vert \nabla u\vert^2} \, + \, g_{\alpha \beta}(u, \!\vartheta) \, d\vartheta^\alpha \otimes d\vartheta^\beta\right)=\frac{du \otimes du}{(1-u)^4} \, + \, \frac{c_{\alpha\beta}(\theta)}{(1-u)^2}\, d\vartheta^\alpha \otimes d\vartheta^\beta\,.
$$
Notice that the vanishing of term $\mathrm{R}\,+\,2\,\mathrm{div}(X)\,-\, (1/2)\,\vert X\vert^2$ in formula~\eqref{feq1} implies that the scalar curvature $\widetilde{\mathrm{R}}$ of the metric $\widetilde{g}$ is zero everywhere. Therefore, by considering the function $t=1/(1-u)$ and by observing that the metric $\widetilde{g}$ can be written as 
$$
\widetilde{g}=dt\otimes dt+t^2 c_{\alpha\beta}(\theta)\,d\vartheta^\alpha \otimes d\vartheta^\beta\,,
$$
the scalar curvature of each level set of the function $t$ with respect to the metric $c_{\alpha\beta}(\theta)\,d\vartheta^\alpha \otimes d\vartheta^\beta$ is identically equal to $2$. Since each level set of the function $t$, which is also a level set of the function $u$, is diffeomorphic to a $2$--sphere, we conclude that each of them is isometric to $(\SSS^2, g_{\SSS^2})$ and that $(M,\widetilde{g}) $ is isometric to $(\R^3, g_{\mathrm{eucl}})$.
\end{proof}

\section{Proof of Theorem~\ref{GenPMT}}\label{XPMTsection}
In Section~\ref{Greensection}, we established the existence of the minimal positive Green's function $\mathcal{G}_o$ for the operator $\mathcal{L}_X=\Delta-\,(1/2) \,\nabla_X$ with a pole at some point $o\in M$, in a complete asymptotically flat $3$--manifold $(M,g)$ of order $\tau\in (1/2,1)$, under the assumption that $X$ is a smooth admissible vector field of $M$. We also proved that:
\begin{itemize}
\item[$\circ$] For any normal coordinate system $(x^1,x^2,x^3)$ centered at $o$ and defined on an open ball $B_{r_o}(o)$, there exist constants $b_{i_1}^{(0)}$, $a^{(1)}$, $c_{i_1 i_2}^{(1)}$, $e_{i_1 i_2 i_3 i_4 }^{(1)}$, $b_{i_1}^{(2)}$, $d_{i_1 i_2 i_3 }^{(2)}$, $f_{i_1 i_2 i_3 i_4 i_5 }^{(2)}$, $a^{(3)}$, $c_{i_1 i_2}^{(3)}$, $e_{i_1 i_2 i_3 i_4 }^{(3)}$, $h_{i_1 i_2 i_3 i_4 i_5 i_6}^{(3)}$ and $\,l_{i_1 i_2 i_3 i_4 i_5 i_6 i_7 i_8}^{(3)}$, depending on the coefficients of the metric $g$ and the vector field $X$ and there exists a function $f\in C^{2}(M)$ such that
\begin{align*}
4\pi \mathcal{G}_o&=\frac{1}{\vert x\vert}+b_{i_1}^{(0)}\frac{x^{i_1}}{\vert x\vert}+\vert x\vert\Big[a^{(1)}+c_{i_1 i_2}^{(1)}\frac{x^{i_1} x^{i_2}}{\vert x\vert^2}+e_{i_1 i_2 i_3 i_4 }^{(1)}\frac{x^{i_1}x^{i_2} x^{i_3} x^{i_4}}{\vert x\vert^4}\Big]\\
&\quad \ +\vert x\vert^2\Big[b_{i_1}^{(2)}\,\frac{x^{i_1}}{\vert x\vert}+d_{i_1 i_2 i_3 }^{(2)}\frac{x^{i_1}x^{i_2} x^{i_3} }{\vert x\vert^3}+f_{i_1 i_2 i_3 i_4 i_5 }^{(2)}\frac{x^{i_1}x^{i_2} x^{i_3} x^{i_4} x^{i_5}}{\vert x\vert^5}\Big]\\
&\quad \ +\vert x\vert^3\Big[a^{(3)}+c_{i_1 i_2}^{(3)}\frac{x^{i_1} x^{i_2}}{\vert x\vert^2}+e_{i_1 i_2 i_3 i_4 }^{(3)}\frac{x^{i_1}x^{i_2} x^{i_3} x^{i_4}}{\vert x\vert^4}+h_{i_1 i_2 i_3 i_4 i_5 i_6}^{(3)}\frac{x^{i_1}x^{i_2} x^{i_3} x^{i_4} x^{i_5} x^{i_6}}{\vert x\vert^6}\\
&\quad\quad\quad \quad\ +l_{i_1 i_2 i_3 i_4 i_5 i_6 i_7 i_8}^{(3)}\frac{x^{i_1}x^{i_2} x^{i_3} x^{i_4} x^{i_5} x^{i_6} x^{i_7} x^{i_8}}{\vert x\vert^8}\Big] + f,
\end{align*}
in $B_{r_o/4}(o)$. It then follows that
\begin{align}
\Big| \mathcal{G}_{o} - \frac{1}{4\pi r}\Big|&\, = \, o (r^{-1})\,,
\label{eq11}
\\
\Big\vert \nabla \mathcal{G}_{o}+\frac{1}{4\pi r^2}\,\nabla r\Big\vert& \,= \, o(r^{-2})\,,
\label{eq12}
\\
\!\!\!\!\!\!\!\!\!\!\!\!\!\!\!\!\!\!\!\!\Big\vert \nabla d\mathcal{G}_{o}-\frac{1}{4\pi r^2}\,\Big(\,\frac{2}{r} \, dr\otimes dr- \nabla d r\Big)\Big\vert& \, = \, o(r^{-3})\,.
\label{eq13}
\end{align}
where $r$ denotes the distance from the pole $o$. In particular, the function $\mathcal{G}_o$ has no critical points in a neighborhood of the point $o$.
\item[$\circ$] For any asymptotically flat coordinate chart $\big(U,(x^1,x^2,x^3)\big)$ of $(M,g)$, there is a positive constant ${A}$ such that
\begin{equation}\label{eq17}
\mathcal{G}_o=\frac{\ {A}\ }{\vert x\vert}+O_{2}(\vert x\vert^{-1-\tau})\,.
\end{equation}
Thus, the function $\mathcal{G}_o$ has no critical points also ``going to infinity''.
\end{itemize}

In Section~\ref{monotonicity formula section}, we considered the function $u\,=\,1-4\pi \mathcal{G}_o$, defined the function $F:(0, + \infty)\to \R$ as
 \begin{equation}
 F(t)\ = \ 4\pi t \ +\ \,t^3\!\!\!\int\limits_{\{u=1-\frac{1}{t}\}}\!\!\!\!\! \vert \nabla u \vert^{2} \ d\mathcal{H}^2 \, \, - \ t^2 \!\!\! \int\limits_{\{u=1-\frac{1}{t}\}}\!\!\!\!\!\vert \nabla u \vert\, \mathrm{H} \ d\mathcal{H}^2 
 \, \, + \ t^2 \!\!\! \int\limits_{\{u=1-\frac{1}{t}\}}\!\!\!\!\! g(X,\nabla u) \ d\mathcal{H}^2 \,,
 \end{equation}
 where $\mathrm{H}$ is the mean curvature on the regular part of the level set $\{u=1-1/t \}\setminus \{\vert \nabla u \vert=0\}$ computed with respect to $\nu={\nabla u}/{\vert \nabla u \vert}$ and showed that
 \begin{equation}
 0 < s \leqslant t < +\infty \quad\text{and}\quad\text{$1- 1/s$, $1- 1/t$ are regular values of $u$} \quad \implies\quad F(s) \, \leqslant \, F(t) \, ,
 \end{equation}
 if 
 $$
 \mathrm{R}_X^{(k)}=\mathrm{R}+2\,\mathrm{div}(X)-(1+1/k)\, \vert X\vert^2\, \geqslant\,0\,,
 $$
 for some $k\in \R\setminus (-2,0]$ and the second integral homology $H_2(M;\Z)$ does not contain any spherical class.\\
By combining these results, we infer that
\begin{equation}\label{consequencemonotonicityformula}
\lim\limits_{t\to 0^+} F(t) \leqslant \lim\limits_{t\to +\infty} F(t)\,.
\end{equation}

We now analyze these two limits.

\begin{lemma}\label{Lemmalimittozero}
Under the conditions~\eqref{eq11},~\eqref{eq12},~\eqref{eq13} and the assumption $\mathcal{G}_o\to 0$ at infinity, we have
\begin{equation}\label{limittozero}
\lim_{t\to 0^+} F(t) \,=\,0\,.
\end{equation}
\end{lemma}
\begin{proof}
We start by recalling that, setting $u(o)=-\infty$, the function $u:M\to [-\infty,1)$ is continuous and proper, as $u\to 1$ at infinity. This fact, together with condition~\eqref{eq12}, implies the existence of some $t\in (0,t_0)$ such that the open set $\{u<1-1/t_0\}$ contains no critical points of $u$. Thus, $(0,t_0)\subseteq \mathcal{T}$, where the set $\mathcal{T}$ is defined by formula~\eqref{mathcalT}. Then, $F$ is continuously differentiable and nondecreasing on the interval under consideration. This monotonicity guarantees the existence of the limit of $F(t)$ as $t\to 0^+$. We will denote this limit by $\ell$ and in order to compute it, we consider the function
\begin{equation}
\mathcal{F}(t)\,=\,2\pi t^2+\!\!\!\int\limits_{\{u<1-\frac{1}{t}\}}\!\!\!\Bigg[\frac{\vert \nabla u \vert^{2}}{(1-u)^3}+\frac{g(\nabla \vert \nabla u \vert,\nabla u)}{(1-u)^2}+\frac{1}{2}\frac{g(X,\nabla u)}{(1-u)^2}\Bigg]\,\frac{\vert \nabla u \vert}{(1-u)^2}\,d\mu
\end{equation}
for every $t\in (0,t_0)$. First, we show that the function $\mathcal{F}$ is well--defined, moreover, the following argument will also imply that $\lim_{t\to 0^+}\mathcal{F}(t)/t=0$. 
We observe that the decay conditions~\eqref{eq11},~\eqref{eq12} and~\eqref{eq13} yield 
\begin{equation}\label{eq14}
\frac{C_1}r \leqslant 1-u \leqslant\frac{C_2}r\,,
\qquad\quad
\frac{C_3}{r^2} \leqslant\vert \nabla u \vert \leqslant \frac{C_4}{r^2}\,,
\qquad\quad
\vert \nabla du \vert \leqslant \frac{C_5}{r^3}\,,
\end{equation}
for some positive constants $C_{i}>0$, $i=1, \ldots, 5$. Therefore, the open set $\{u<1-1/t\}$ is contained in the ball $B_{C_2 t}(o)$ and 
$$\frac{\vert \nabla u \vert}{(1-u)^2} \leqslant \frac{C_4}{C_1^2}\,,\quad\quad \quad\bigg\vert\frac{g(\nabla \vert \nabla u \vert,\nabla u)}{(1-u)^3}\bigg\vert\leqslant \frac{C_5}{C_1^3}\,,\quad\quad \quad\bigg\vert\frac{g(X,\nabla u)}{(1-u)^2}\bigg\vert \leqslant \frac{C_4}{C_1^2}\,\vert X\vert\,.$$
Thus, we get
\begin{align}
&\int\limits_{\{u<1-\frac{1}{t}\}}\Bigg\vert\frac{\vert \nabla u \vert^{2}}{(1-u)^3}+\frac{g(\nabla \vert \nabla u \vert,\nabla u)}{(1-u)^2}+\frac{1}{2}\frac{g(X,\nabla u)}{(1-u)^2}\Bigg\vert\,\frac{\vert \nabla u \vert}{(1-u)^2}\,d\mu\\
&\quad\quad \quad \leqslant \int\limits_{B_{C_2 t}(o)}\Bigg\vert\frac{\vert \nabla u \vert^{2}}{(1-u)^3}+\frac{g(\nabla \vert \nabla u \vert,\nabla u)}{(1-u)^2}+\frac{1}{2}\frac{g(X,\nabla u)}{(1-u)^2}\Bigg\vert\,\frac{\vert \nabla u \vert}{(1-u)^2}\,d\mu\\
&\quad\quad \quad \leqslant \Big(\frac{C_4^2}{C_1^4}+ \frac{C_5}{C_1^3}\Big)\,\frac{C_4}{C_1^2} \int\limits_{B_{C_2 t}(o)}\!\!\! (1-u)\,d\mu\,+\,\frac{C_4^2}{C_1^4}\,\bigg(\sup_{B_{C_2 t_0}(o)}\!\!\vert X\vert\bigg)\,\mathrm{Vol}\big(B_{C_2 t}(o)\big)\,.\label{eq16}
\end{align}
Now, the estimates $\mathrm{Vol}\big(B_{r}(o)\big)=\frac{4}{3}\pi r^3 +O(r^{5})$ and $\mathrm{Area}\big(\partial B_{r}(o)\big)=4\pi r^2 +O(r^4)$ imply
\begin{equation}\label{eq15}
\mathrm{Vol}\big(B_{r}(o)\big) \leqslant \frac{4}{3}C_6\pi r^3 \qquad\text{ and }\qquad \mathrm{Area}\big(\partial B_{r}(o)\big) \leqslant 4C_7\pi r^2\,,
\end{equation}
for some positive constants $C_{6}, C_{7}>0$ and for any $r\in (0,r_0)$, where $r_0>0$ is sufficiently small. Therefore, by applying the coarea formula in inequality~\eqref{eq16} and using estimates~\eqref{eq14} and~\eqref{eq15}, we obtain
\begin{align*}
\int\limits_{\{u<1-\frac{1}{t}\}}\!\Bigg\vert\frac{\vert \nabla u \vert^{2}}{(1-u)^3}+&\,\frac{g(\nabla \vert \nabla u \vert,\nabla u)}{(1-u)^2}+\frac{1}{2}\frac{g(X,\nabla u)}{(1-u)^2}\Bigg\vert\,\frac{\vert \nabla u \vert}{(1-u)^2}\,d\mu\\
\leqslant&\, 2\pi C_2^3C_7\Big(\frac{C_4^2}{C_1^4}+ \frac{C_5}{C_1^3}\Big)\,\frac{C_4}{C_1^2}\,t^2+\,C_6 C_2^3 \pi\bigg(\sup_{B_{C_2 t_0}(o)}\!\!\vert X\vert\bigg)\,\frac{4C_4^2}{3C_1^4}\,t^3\,,
\end{align*}
for all $t\in (0,t_0)$, possibly passing to a smaller $t_0>0$.
The reason for introducing the function $\mathcal{F}(t)$ lies in the fact that 
$$\mathcal{F}(t)=\int_{0}^t F(\tau)\,d\tau\,,$$
by the coarea formula. Thus, by using that $F$ is continuously differentiable in the set $\mathcal{T}$, it follows that $\mathcal{F}$ is of class $C^2$
on the interval $(0,t_0)$ and $\mathcal{F}'(t)=F(t)$ for every $t\in (0,t_0)$. Then, the generalized version of de~l'H\^opital's rule in~\cite[Theorem~II]{taylor1} gives
\begin{equation}
\ell=\limsup\limits_{t\to 0^+} F(t) \geqslant \limsup\limits_{t\to 0^+}\frac{\mathcal{F}(t)}{t}=0= \liminf\limits_{t\to 0^+}\frac{\mathcal{F}(t)}{t} \geqslant \liminf\limits_{t\to 0^+} F(t)=\ell\,,
\end{equation}
which implies $\ell=0$.
\end{proof}

By formula~\eqref{eq17}, we have a positive constant $B$ such that
\begin{equation}\label{eq18}
u=1-\frac{\ {B}\ }{\vert x\vert}+O_{2}(\vert x\vert^{-1-\tau})\,,
\end{equation}
as $u=1-4\pi \mathcal{G}_o$.

\begin{lemma}\label{lemM}
Under the assumption of Theorem~\ref{GenPMT}, we have
\begin{equation}\label{limtoinfinityF}
\lim_{t\to+\infty} F(t) \, \leqslant \,8\pi m_{X}/{B}\,.
\end{equation}
\end{lemma}

\begin{proof}
In order to show the claim, we will proceed similarly to~\cite[Section 3]{Or_1}, but we will also use some results contained in~\cite[Section 2]{Miao2022}. The reason lies in the fact that we need to obtain area estimates on the level sets $\{u=1-(1/t)\}$, for $t\in (0,+\infty)$ sufficiently large, in a different way with respect to~\cite[Section 3]{Or_1}, where we used the divergence form of the equation satisfied by the function $u$. We start by observing that there is a $T_0\in (0,+\infty)$ such that the set $\{u \geqslant 1-(1/T_0)\}$ is contained in $U$, where $\big(U,(x^1,x^2,x^3)\big)$ is an asymptotically flat coordinate chart of $(M,g)$ that we fix from now on. Possibly passing to a larger $T_0$, by identity~\eqref{eq17}, the set $\{u \geqslant 1-(1/T_0)\}$ does not contain critical points of $u$. Now, we notice that all the level sets $\{u=1-(1/t)\}$, with $t\in [T_0,+\infty)$, are connected. Indeed, since $M$ has a single end (being asymptotically flat), every open set $\{u>1-(1/t)\}$, with $t\in\mathcal{T}$, is connected. The connectedness of these level sets then follows since the number of the connected components of $\{u=1-(1/t)\}$, with $t\in [T_0,+\infty)$, is the same as $\{u=1-(1/T_0)\}$, which is in turn equal to the number of the connected components of $\{u>1-(1/T_0)\}$. Moreover, by assumption~\eqref{eq1} and asymptotic expansion~\eqref{eq18}, we deduce
\begin{align}
\vert \nabla u\vert&=\frac{B}{\ \vert x\vert^{2}}+O\big(\vert x\vert^{-2-\tau}\big)\,, \label{eq19}\\
(\nabla du)_{ij}&=-\frac{{B}}{\vert x \vert^{3}}\bigg(\frac{3}{\vert x \vert^{2}}x^{i}x^{j}-\delta_{ij}\bigg)+O\big(\vert x\vert^{-3-\tau}\big)\,. \label{eq20}
\end{align}
Therefore, possibly passing to a larger $T_0$, the mean curvature $\mathrm{H}$ of $\{u=1-1/t \}$, for any $t\in [T_0,+\infty)$, which is computed with respect to the unit normal $\nu={\nabla u}/{\vert \nabla u \vert}$, satisfies
\begin{equation}\label{eqf48bis}
{\mathrm{H}}\,= \,\frac{2}{\vertl x\vertr}+O(\vert x\vert^{-1-\tau}) \,.
\end{equation}
This implies that the level sets $\{u=1-1/t \}$ are outer area--minimizing, as well as the coordinate spheres. By combining this fact with the estimate 
\begin{equation}\label{eq23}
\vert x\vert = {B} t+O(t^{1-\tau})
\end{equation}
on $\{u=1-1/t\}$ as $t\to +\infty$, we can conclude that 
\begin{equation}\label{eq22}
\mathrm{Area}\big(\{u=1-1/t \}\big)=4\pi {B}^2 t^2+O(t^{2-\tau})
\end{equation}
(see~\cite[pages 357--358]{Miao2022}).\\
We now use this area estimate to show that the last term in the following equivalent form of the function $F$,
\begin{equation}\label{eq21}
F(t)=\frac{t}{4}\bigg(16\pi \ - \!\!\!\int\limits_{\{u=1-\frac{1}{t}\}}\!\!\!\mathrm{H}^{2}\,d\mathcal{H}^2 +4\!\!\int\limits_{\{u=1-\frac{1}{t}\}}\!\!\!\!\! \frac{\ g(X,\nabla u)\ }{(1-u)} \,d\mathcal{H}^2\bigg)+\frac{t}{4}\!\!\int\limits_{\{u=1-\frac{1}{t}\}}\!\!\!\bigg(\frac{2\vert \nabla u \vert}{1-u}\,-\,\mathrm{H}\bigg)^{\!2} \,d\mathcal{H}^2\,,
\end{equation}
converges to zero as $t\to+\infty$. Indeed, estimates~\eqref{eq18},~\eqref{eq19} and~\eqref{eqf48bis}, yield
$$
\frac{2\vert \nabla u \vert}{1-u}\,-\,\mathrm{H}=O(\vert x\vert^{-1-\tau})\,.
$$
By combining this and equality~\eqref{eq23}, from formula~\eqref{eq22} we then have that the integral of the last term in equation~\eqref{eq21} is of the form $O(t^{-2\tau})$, hence it goes to zero, as $t\to+\infty$, since $\tau>1/2$.\\
We now set
\begin{equation}
{XM}(t)=\frac{t}{4}\bigg(16\pi \ - \!\!\!\int\limits_{\{u=1-\frac{1}{t}\}}\!\!\!\mathrm{H}^{2}\,d\mathcal{H}^2 +4\!\!\int\limits_{\{u=1-\frac{1}{t}\}}\!\!\!\!\! \frac{\ g(X,\nabla u)\ }{(1-u)} \,d\mathcal{H}^2\bigg)
\end{equation}
and we analyze its behavior at infinity. To this aim, it is convenient to use either an overbar or a subscript $e$ whenever a quantity is computed with respect to the Euclidean metric $e$ in the asymptotically flat coordinate chart $\big(U,(x^1,x^2,x^3)\big)$, that is, $e=\delta_{ij}dx^i\otimes dx^j$. Conversely, we agree that quantities computed with respect to the metric $g$ will be denoted without any additional symbol. Moreover, the covariant derivative with respect to $g$ will be denoted by $\na$, whereas the symbol $\D$ will indicate the Euclidean covariant derivative. Finally, we denote by $\gamma$ the $(0,2)$--symmetric tensor given by 
$$
g-e=(g_{ij}-\delta_{ij})\,dx^i\otimes dx^j
$$
and we observe that, according to formula~\eqref{eq1}, there holds
\begin{equation}
g^{ij} \,= \,\delta^{ij}-\gamma^{ij}+\, O(|x|^{-2\tau}) \,\label{exp2} \,,
\end{equation}
where $\gamma^{ij} = \delta^{i \ell} \delta^{k j} \gamma_{k \ell}$.
Our first task is to obtain an expansion for the mean curvature $\HHH$ of $\Sigma_t=\{u=1-1/t\}$ in terms of its Euclidean mean curvature $\overline{\HHH}$. 
We start by observing that the unit normal vector $\nu=\na u/|\na u|$ is related to the Euclidean one $\overline{\nu} \,= \,\D u/|\D u|_e$ through the formula
\begin{equation}
\nu^i \,= \,\Big( 1 + \frac{\gamma(\overline{\nu},\overline{\nu})}{2}\Big) \,\overline{\nu}^i - \,\gamma^i_k\,\overline{\nu}^k \,+\, O(|x|^{-2\tau}) \,,
\end{equation}
where $\gamma^i_k = \delta^{i j} \gamma_{j k}$.
Since the mean curvatures are computed as
\begin{equation}
\HHH \,= \,\left( g^{ij} - \nu^i \nu^j \right) \,\frac{(\na du)_{ij}}{|\na u|} \qquad \hbox{and} \qquad \overline{\HHH} \,= \,\left( \delta^{ij} - \overline{\nu}^i \overline{\nu}^j \right)\frac{ (\D du)_{ij} }{|\D u|_e} \,, 
\end{equation}
respectively, by noticing that $|\D du|_e/ |\D u|_e= O (|x|^{-1})$ by expansion~\eqref{eq18} and setting $\eta^{i j} = \delta^{ij} - \overline{\nu}^i \overline{\nu}^j $, we then arrive at 
\begin{equation}\label{eqHH}
\HHH \,= \,\Big( 1 + \frac{\gamma(\overline{\nu},\overline{\nu})}{2}\Big) \,\overline{\HHH} - \eta^{ij} \Big( \pa_j g_{ik} - \frac{1}{2} \pa_k g_{ij} \Big) \overline{\nu}^k - \eta^{ik} \eta^{j \ell} \gamma_{k \ell} \frac{ (\D du)_{ij}}{|\D u|_e} + O(|x|^{-1-2\tau}) \,,
\end{equation}
which implies 
$$
\HHH^2 \,= \,\big( 1 + {\gamma(\overline{\nu},\overline{\nu})} \big) \,\overline{\HHH}^2 - 2 \overline{\HHH} \,\eta^{ij} \Big( \pa_j g_{ik} - \frac{1}{2} \pa_k g_{ij} \Big) \overline{\nu}^k - 2\overline{\HHH} \,\eta^{ik} \eta^{j \ell} \gamma_{k \ell} \frac{ (\D du)_{ij}}{|\D u|_e} + O(|x|^{-2-2\tau}) \,.
$$
As the metric induced on $\Sigma_t$ by $g$ can be written as $g - \frac{du \otimes du}{|\na u|^2}$, the area element can be expressed as 
\begin{equation}
d\mathcal{H}^2 \,= \,\Big[1+\frac{1}{2}\,\eta^{ij}\gamma_{ij}+ O(\vert \xx x \xx \vert^{-2\tau})\Big] \,d\overline{\mathcal{H}}^2\,. \label{exp4} 
\end{equation}
Putting all together, the ``Willmore energy integrand'' then satisfies
\begin{align}
\HHH^2 \,d\mathcal{H}^2\,= \,\Big[\Big( 1&\, + \gamma(\overline{\nu},\overline{\nu}) + \frac{\eta^{ij}\gamma_{ij}}{2} \Big) \,\overline{\HHH}^2 - \,2 \overline{\HHH} \,\eta^{ij} \Big( \pa_j g_{ik} - \frac{1}{2} \pa_k g_{ij} \Big) \overline{\nu}^k\\
&\, - \,2\overline{\HHH}\,\eta^{ik} \eta^{j \ell} \gamma_{k \ell} \frac{ (\D du)_{ij}}{|\D u|_e} + O(|x|^{-2-2\tau}) \Big] \,d\overline{\mathcal{H}}^2\,.
\quad\qquad \label{eq:will}
\end{align}
Now, by means of formula~\eqref{eq18} again, we have
\begin{align}
\overline{\nu}^i&=\frac{x^i}{\vert x\vert}+O(|x|^{-\tau})\,,\label{eq24}\\
\frac{ (\D du)_{ij}}{|\D u|_e}& \,= \,\frac{1}{\vertl x\vertr}\Big(\delta_{ij}-3\,\delta_{ik} \delta_{j \ell}{\overline{\nu}}^{k}\,{\overline{\nu}^{\ell}}+ O(|x|^{-\tau})\,\Big)\,,\label{eqf47}
 \end{align}
which imply
\begin{equation}\label{eqf48}
\overline{{\mathrm{H}}}\,= \,\frac{2}{\vertl x\vertr}\bigl(1+O(|x|^{-\tau})\bigr)
\end{equation}
and, in turn, 
\begin{equation}\label{eql49}
2\overline{\HHH} \,\eta^{ik} \eta^{j \ell} \gamma_{k \ell} \frac{ (\D du)_{ij}}{|\D u|_e} \,= \,\frac{4}{|x|^2} \eta^{k \ell} \gamma_{k \ell} +O(|x|^{-2-\tau})\,.
\end{equation}
Plugging these two last equalities in formula~\eqref{eq:will}, we obtain
\begin{equation}
\HHH^2 \,d\mathcal{H}^2\ = \ \left[ \,\overline{\HHH}^2 \,+ \,\frac{4}{|x|^2} \gamma(\overline{\nu} ,\overline{\nu}) \,- \,\frac{2}{|x|^2} \eta^{i j} \gamma_{i j} \,- \,\frac{4}{|x|} \eta^{ij} \Big( \pa_j g_{ik} - \frac{1}{2} \pa_k g_{ij} \Big) \overline{\nu}^k + O(|x|^{-2-2\tau}) \right] \,d\overline{\mathcal{H}}^2 \,.
\end{equation}
To proceed, we now claim that 
\begin{equation}
\label{claimdiv}
\frac{4}{|x|^2} \gamma(\overline{\nu} , \overline{\nu}) \,- \,\frac{2}{|x|^2} \eta^{i j} \gamma_{i j} \,= \,\frac{2}{|x|} \bigl( \eta^{ij} \pa_i g_{jk} \overline{\nu}^k - \,\overline{\mathrm{div}}_{\Sigma_{t}} \omega^\top\,\bigr) \,+ \,O(|x|^{-2-2\tau})\,,
\end{equation}
where $\omega$ is the differential $1$--form defined by $\omega = \gamma_{jk} \overline{\nu}^k {d} x^j$ and $\omega^\top$ denotes its tangential component. To prove the claim, let us first observe that
$$
\omega = \omega^\top +\omega(\overline{\nu})\frac{du}{\ \vert \D u\vert_e} = \omega^\top + \gamma(\overline{\nu},\overline{\nu})\frac{du}{\ \vert \D u\vert_e}\qquad\text{ and }\qquad
\pa_i \overline{\nu}^k \,= \,\eta^{k \ell} \frac{ (\D du)_{i\ell}}{|\D u|_e} \,.
$$
By means of expansions~\eqref{eqf48} and~\eqref{eql49}, we then observe that
\begin{align}
\eta^{ij} \pa_i g_{jk} \overline{\nu}^k =\eta^{ij} \pa_i \gamma_{jk} \overline{\nu}^k & = \,\eta^{ij} \pa_i \omega_j \,- \,\eta^{ik} \eta^{j \ell} \gamma_{k \ell} \frac{ (\D du)_{ij}}{|\D u|_e}\\
& = \,\overline{\mathrm{div}}\,\omega \,- \,\pa_i \omega_j \overline{\nu}^i \overline{\nu}^j \,- \,\frac{1}{|x|} \eta^{ij} \gamma_{ij} + O(|x|^{-2-2\tau})\\
& = \,\overline{\mathrm{div}}_{\Sigma_t} \omega^\top \,+ \,\gamma(\overline{\nu},\overline{\nu}) \overline{\HHH} - \,\frac{1}{|x|} \eta^{ij} \gamma_{ij} + O(|x|^{-2-2\tau})\\
& = \,\overline{\mathrm{div}}_{\Sigma_t}\omega^\top \,+ \,\frac{2}{|x|}\gamma(\overline{\nu},\overline{\nu}) - \,\frac{1}{|x|} \eta^{ij} \gamma_{ij} +O(|x|^{-2-2\tau})\,,
\end{align}
which yields formula~\eqref{claimdiv}. As a consequence,
\begin{align}
\HHH^2 \,d\mathcal{H}^2=& \left[ \,\overline{\HHH}^2\! - \frac{2}{|x|}\, \overline{\mathrm{div}}_{\Sigma_{t}} \omega^\top\!\! + \frac{2}{|x|} \eta^{ij} \pa_i g_{jk} \overline{\nu}^k\! - \frac{4}{|x|} \eta^{ij} \pa_j g_{ik}\overline{\nu}^k\! + \frac{2}{|x|}\eta^{ij} \pa_k g_{ij} \overline{\nu}^k+ O(|x|^{-2-2\tau}) \right] \,d\overline{\mathcal{H}}^2\\
= \ & \left[ \,\overline{\HHH}^2 - \frac{2}{|x|}\, \overline{\mathrm{div}}_{\Sigma_{t}} \omega^\top - \frac{2}{|x|} \delta^{ij} \big( \pa_i g_{jk} - \pa_k g_{ij} \big) \overline{\nu}^k + O(|x|^{-2-2\tau}) \right] \,d\overline{\mathcal{H}}^2 \,. \label{eq:nice}
\end{align}
In view of equality~\eqref{eq23}, formula~\eqref{eq:nice} becomes
\begin{equation}
\HHH^2 \,d\mathcal{H}^2\ = \ \left[ \,\overline{\HHH}^2 - \frac{2}{{B}t}\,\overline{\mathrm{div}}_{\Sigma_{t}} \omega^\top - \frac{2}{{B}t} \delta^{ij} \big( \pa_i g_{jk} - \pa_k g_{ij} \big) \overline{\nu}^k + O(t^{-2-2\tau}) \right] \,d\overline{\mathcal{H}}^2 \quad \text{on $\Sigma_t$\,.}\label{eq:verynice}
\end{equation}
Now, we observe that
\begin{equation}
\frac{\ g(X,\nabla u)\ }{(1-u)} \,d\mathcal{H}^2=\bigg[\delta_{ij}X^i \frac{x^j}{\vert x\vert^2}+O(|x|^{-2-2\tau})\bigg] \,d\overline{\mathcal{H}}^2=\bigg[\,\frac{1}{|x|}\delta_{ij}X^i \overline{\nu}^j+O(|x|^{-2-2\tau})\bigg] \,d\overline{\mathcal{H}}^2 \,.
\end{equation}
as a consequence of estimates~\eqref{eq18},~\eqref{exp4},~\eqref{eq24} and the fact that $X=O(\vert x\vert^{-\tau-1})$. Similarly as before, equality~\eqref{eq23} then gives
\begin{equation}
\frac{\ 4g(X,\nabla u)\ }{(1-u)} \,d\mathcal{H}^2=\bigg[\,\frac{4}{{B}t}\delta_{ij}X^i \overline{\nu}^j+O(t^{-2-2\tau})\bigg] \,d\overline{\mathcal{H}}^2 \quad \text{on $\Sigma_t$\,.}\label{eq:verynicebis}
\end{equation}
By putting equality~\eqref{eq22} in the above identities~\eqref{eq:verynice} and~\eqref{eq:verynicebis}, we obtain
\begin{align*}
{XM}(t)=&\,\frac{t}{4}\bigg(16\pi \ - \!\int\limits_{\Sigma_t}\mathrm{H}^{2}\,d\mathcal{H}^2 +4 \int\limits_{\Sigma_t}\frac{\ g(X,\nabla u)\ }{(1-u)} \,d\mathcal{H}^2\bigg)\\
=&\,\frac{t}{4}\,\xx\,\bigg( 16\xx\pi \ - \!\int\limits_{\Sigma_t}\overline{\mathrm{H}}^{2} \,d\overline{\mathcal{H}}^2 \bigg) + \,\frac{1}{2{B}} \int\limits_{\Sigma_t} \overline{\mathrm{div}}_{\Sigma_{t}}\omega^\top\,d\overline{\mathcal{H}}^2\\
&\,+ \,\frac{1}{2{B}}\Bigg( \,\int\limits_{\Sigma_t} 
\delta^{ij} \big( \pa_i g_{jk} - \pa_k g_{ij} \big) \overline{\nu}^k \,d\overline{\mathcal{H}}^2 +\int\limits_{\Sigma_t} 2\delta_{ij}X^i \overline{\nu}^j\,d\overline{\mathcal{H}}^2\Bigg)+ O(t^{1-2\tau}) \,.
\end{align*}
The first summand is nonpositive by the {\em Euclidean Willmore inequality} (see~\cite{Will}), the second summand vanishes by the divergence theorem and the third summand tends to $8 \pi m_{X}/{B}$, as $t \to + \infty$ (the last term goes to zero). In conclusion, we have
\begin{equation}
\lim_{t \to + \infty} F(t) \, \leqslant \,\limsup_{t\to +\infty}\, XM(t)\, \leqslant \,8\pi m_X/{B}\,.
\end{equation}
\end{proof}

\begin{remark}
If $g_{ij}=\delta_{ij} + O_{3}(\vert x \vert ^{-\tau})$ and $X^i=O_2(\vert x \vert ^{-1-\tau})$ in an asymptotically flat coordinate chart $\big(U,(x^1,x^2,x^3)\big)$, then it turns out that the function $u$ satisfies
\begin{equation}
u=1-\frac{\ {B}\ }{\vert x\vert}+O_{3}(\vert x\vert^{-1-\tau})\,.
\end{equation}
Then, by the argument of the proof of Theorem~2.1 in~\cite{Miao2022}, the inequality~\eqref{limtoinfinityF} can be shown to be to an equality.
\end{remark}

By Lemmas~\ref{Lemmalimittozero} and~\ref{lemM}, from inequality~\eqref{consequencemonotonicityformula}, we conclude that
$m_X \geqslant 0$, as ${B}>0$, hence the first part of Theorem~\ref{GenPMT} is proved. The rigidity part is a direct consequence of Proposition~\ref{Constantcase}.

\begin{remark}
Theorem~\ref{GenPMT} holds under slightly weaker assumptions on the decay of the metric $g$ and of the vector field $X$. More precisely, it is still true provided that the metric $g$ belongs to $C^{1,\alpha}_{-\tau}(M)$ and the vector field $X$ lies in $C^{0,\alpha_0}_{-1-\tau_0}(M)$, for some $\alpha,\alpha_0\in (0,1)$ and $\tau, \tau_0>1/2$.\\
We recall that a smooth function $u$ on $M$ lies in the {\em weighted H\"{o}lder space} $C^{k,\alpha}_{s}(M)$, if the {\em weighted norm} 
$$
\Vert u\Vert_{C^{k,\alpha}_{s}(M)}=\sum_{i=1}^{k} \sup\limits_{M} \vert r^{i-s}\nabla^i u\vert+\sup\limits_{x\in M}\Bigg( r^{k+\alpha-s}\!\! \sup_{\substack{y,z\in B_{r/2}(x)\\x\neq y}} \frac{\vert \nabla^k u (y)-\nabla^{k} u (z)\vert}{d(y,z)^\alpha}\Bigg)
$$
is finite, where $r$ is a smooth positive function on $M$ such that $r = \vert x\vert$ in an asymptotically flat coordinate chart and the quantity $\vert \nabla^k u (y)-\nabla^{k} u (z)\vert$ can be defined by using parallel translation along
a minimizing geodesic connecting $y$ to $z$. If $E$ is a smooth vector bundle over $M$ equipped with a bundle metric and connection, then the spaces of sections $C^{k,\alpha}_{s}(E)$ are defined analogously.
\end{remark}

\begin{remark}
Under the assumption $(2)$ in Theorem~\ref{GenPMT}, if we also assume that $X$ is a divergence free vector field on $M$ and that $m_X=0$, then $(M,g)$ is isometric to $(\R^3, g_{{\mathrm{eucl}}})$. Indeed, by the proof of Proposition~\ref{Constantcase}, we know that $X$ is the gradient of the function
$$f=-\log\left(\frac{\vert \nabla u\vert^2}{(1-u)^4}\right)\,,$$
which is smooth on all of $M$ (notice that it converges at infinity by formula~\eqref{eq18}). Therefore, the assumption that $X$ has zero divergence implies that the function $f$, in addition to being bounded, is also harmonic on the whole $M$, hence, by the maximum principle, it is constant. Consequently, $X$ vanishes everywhere and the rest of the claim follows as point $(a)$ in Proposition~\ref{Constantcase}.
\end{remark}

\begin{remark}
An analogue of Theorem~\ref{GenPMT} holds for manifolds with boundary. Let $(M,g)$ be a connected, orientable, complete, asymptotically flat $3$--manifold with a compact and connected boundary and let $X$ be an admissible smooth vector field on $M$. Assume that $\mathrm{R}+2\,\mathrm{div}(X)\in L^{1}(M,g)$, $\,\mathrm{R}_X^{(k)}=\mathrm{R}+2\,\mathrm{div}(X)-(1+1/k)\, \vert X\vert^2\, \geqslant\,0$ for some $k\in \R\setminus (-2,0]$, $\,H_2(M, \partial M;\Z)=0$, then,
\begin{equation}\label{GenPMI}
16\pi -\int\limits_{\partial M} \big[\mathrm{H}-g(X,\nu)\big]^2\, d\mathcal{H}^2 \geqslant 0\,\,\qquad\implies\,\,\qquad m_{X} \geqslant 0\,.
\end{equation}
Furthermore, the following statements are true.
\begin{enumerate}
\item If $\mathrm{R}_X^{(k)} \geqslant 0$ with $k\in \R\setminus [-2,0]$ and $m_X=0$, then $X$ vanishes on $M$ and $(M,g)$ is isometric to $(\R^3,g_{\mathrm{eucl}})$ minus an open ball.
\item If $\mathrm{R}_X^{(-2)} \geqslant 0$ and $m_X=0$, then $X$ is a gradient of a smooth function and $(M,g)$ is conformally isometric to $(\R^3,g_{\mathrm{eucl}})$ minus an open ball.
\end{enumerate}
The proof follows in the same way as the one of Theorem~\ref{GenPMT}. In this case, we consider the function $u$ given by 
$$
\mathcal{L}_Xu=0\,\,\,\text{on $M$,}\quad\quad u=0\,\,\,\text{at $\partial M$}\qquad\text{ and } \qquad u\to1\,\,\, \text{at infinity.}
$$
Therefore, the function $F$, given by formula~\eqref{eq0}, is defined in the interval $[1,+\infty)$. It is nondecreasing in the set $\mathcal{T}$ (see formula~\eqref{mathcalT}) and, comparing its value at $t=1\in \mathcal{T}$ with its limit as $t\to +\infty$, we find that there exists a positive constant ${B}$ such that 
\begin{align}
8\pi m_X/{B}\,&\geqslant\, 4\pi\, +\, \int\limits_{\partial M} \vert \nabla u \vert^{2} \ d\mathcal{H}^2 \,- \,\int\limits_{\partial M}\vert \nabla u \vert\, \mathrm{H} \ d\mathcal{H}^2 
 \, + \, \int\limits_{\partial M} g(X,\nabla u) \ d\mathcal{H}^2\\
 &=\,4\pi\, +\, \int\limits_{\partial M} \left[\,\vert \nabla u \vert\,-\,\frac{\mathrm{H}}{2}\,+\,\frac{1}{2}\,g\left(X,\frac{\nabla u }{\vert \nabla u \vert}\right)\right]^2 d\mathcal{H}^2 \,- \,\frac{1}{4}\int\limits_{\partial M}\,\left[\,\mathrm{H}\,-\,g\left(X,\frac{\nabla u }{\vert \nabla u \vert}\right)\right]^2 d\mathcal{H}^2\\
 &\geqslant\,\frac{1}{4}\,\left[16\pi- \,\int\limits_{\partial M}\!\left[\,\mathrm{H}\,-\,g\left(X,\nu\right)\right]^2 d\mathcal{H}^2\right]\,.
\end{align}
Then, we obtain the conclusion as for Theorem~\ref{GenPMT}.
\end{remark}

\appendix

\section{Proof of the existence of the function \texorpdfstring{$w$}{} in Step~1 of Subsection~\ref{subsecexistenceandasymptoticnearthepole}}\label{proofconstofw}

On the punctured open ball $B_{r_o}(o)\setminus\{o\}$ we consider the function
\begin{align*}
&w=\frac{1}{\vert x\vert}+b_{i_1}^{(0)}\frac{x^{i_1}}{\vert x\vert}+\vert x\vert\Big(a^{(1)}+c_{i_1 i_2}^{(1)}\frac{x^{i_1} x^{i_2}}{\vert x\vert^2}+e_{i_1 \cdots i_4 }^{(1)}\frac{x^{i_1}\cdots x^{i_4}}{\vert x\vert^4}\Big)\\
&\quad\quad+\vert x\vert^2\Big(b_{i_1}^{(2)}\,\frac{x^{i_1}}{\vert x\vert}+d_{i_1 i_2 i_3 }^{(2)}\frac{x^{i_1}x^{i_2} x^{i_3} }{\vert x\vert^3}+f_{i_1\cdots i_5 }^{(2)}\frac{x^{i_1}\cdots x^{i_5}}{\vert x\vert^5}\Big)\\
&\quad\quad +\vert x\vert^3\Big(a^{(3)}+c_{i_1 i_2}^{(3)}\frac{x^{i_1} x^{i_2}}{\vert x\vert^2}+e_{i_1\cdots i_4 }^{(3)}\frac{x^{i_1}\cdots x^{i_4}}{\vert x\vert^4}+h_{i_1 \cdots i_6}^{(3)}\frac{x^{i_1}\cdots x^{i_6}}{\vert x\vert^6}
+l_{i_1\cdots i_8}^{(3)}\frac{x^{i_1}\cdots x^{i_8}}{\vert x\vert^8}\Big)\,.
\end{align*}
Let us determine the constants $b_{i_1}^{(0)}$, $a^{(1)}$, $c_{i_1 i_2}^{(1)}$, $e_{i_1 i_2 i_3 i_4 }^{(1)}$, $b_{i_1}^{(2)}$, $d_{i_1 i_2 i_3 }^{(2)}$, $f_{i_1 i_2 i_3 i_4 i_5 }^{(2)}$, $a^{(3)}$, $c_{i_1 i_2}^{(3)}$, $e_{i_1 i_2 i_3 i_4 }^{(3)}$, $h_{i_1 i_2 i_3 i_4 i_5 i_6}^{(3)}$, $l_{i_1 i_2 i_3 i_4 i_5 i_6 i_7 i_8}^{(3)}$ so that the function $w$ satisfies 
\begin{equation}
\mathcal{L}_X w=h\quad \text{ on} \quad B_{r_o}(o)\setminus\{o\} \,,
\end{equation}
where $h$ is a smooth function in the punctured open ball $B_{r_o}(o)\setminus\{o\}$ that admits a $C^1$--extension on $\overline{B}_{r_o}(o)$, still denoted by $h$. 
We start by observing that
\begin{align*}
\partial_{x^k}\frac{ x^{i_1}}{\vert x\vert}&=\frac{1}{\vert x\vert}\Big[\delta_k^{i_1}-\frac{x^{i_1}x^k}{\vert x\vert^2}\Big],\\
\partial_{x^k}\frac{x^{i_1} x^{i_2}}{\vert x\vert^2}&=\frac{1}{\vert x\vert}\Big[\delta_k^{i_1}\frac{ x^{i_2}}{\vert x\vert}+\delta_k^{i_2}\frac{ x^{i_1}}{\vert x\vert} -2\frac{x^{i_1}x^{i_2}x^k}{\vert x\vert^3}\Big],\\
&\,\cdots\cdot\\
\partial_{x^k}\frac{x^{i_1}x^{i_2} \cdots x^{i_n}}{\vert x\vert^n}&=\frac{1}{\vert x\vert}\Big[\delta_k^{i_1}\frac{x^{i_2} x^{i_3} \cdots x^{i_n}}{\vert x\vert^{n-1}}+\delta_k^{i_2}\frac{x^{i_1} x^{i_3} \cdots x^{i_n}}{\vert x\vert^{n-1}}+\cdots\\
&\qquad\qquad\,\cdots+\delta_k^{i_n}\frac{x^{i_1} x^{i_2} \cdots x^{i_{n-1}}}{\vert x\vert^{n-1}}-n\frac{x^{i_1}x^{i_2}\cdots x^{i_n}x^k}{\vert x\vert^{n+1}}\Big],
\end{align*}
which leads to the following equalities
\begin{align*}
\partial_{x^l}\partial_{x^k}\frac{x^{i_1}x^{i_2} \cdots x^{i_n}}{\vert x\vert^n}&=\frac{1}{\vert x\vert^2}\bigg[\sum_{\substack{r,s =1 \\ r\neq s}}^{n}\delta_k^{i_r} \delta_l^{i_s} \frac{x^{i_1}\cdots\widehat{x^{i_r}} \cdots\widehat{x^{i_s}}\cdots x^{i_n}}{\vert x\vert^{n-2}}\\
&\quad\quad \quad \quad-n \sum_{r=1}^n\delta_k^{i_r} \frac{x^{i_1}\cdots\widehat{x^{i_r}} \cdots x^{i_n}x^l}{\vert x \vert^{n}} -n \sum_{s=1}^n\delta_l^{i_s} \frac{x^{i_1}\cdots\widehat{x^{i_s}} \cdots x^{i_n}x^k}{\vert x\vert^n}\\
&\quad\quad \quad \quad -n \delta_{kl}\frac{x^{i_1}x^{i_2} \cdots x^{i_n}}{\vert x\vert^n} +n(n+2)\frac{x^{i_1}x^{i_2} \cdots x^{i_n} x^k x^l}{\vert x\vert^{n+2}}\bigg],\\
\sum_{k=1}^3\partial_{x^k}\partial_{x^k}\frac{x^{i_1}x^{i_2} \cdots x^{i_n}}{\vert x\vert^n}&\!=\!\frac{1}{\vert x\vert^2}\bigg[\sum_{k=1}^3\sum_{\substack{r,s =1 \\ r\neq s}}^{n}\delta_k^{i_r} \delta_k^{i_s} \frac{x^{i_1}\cdots\widehat{x^{i_r}} \cdots\widehat{x^{i_s}}\cdots x^{i_n}}{\vert x\vert^{n-2}}\!-\!n(n+1)\frac{x^{i_1}x^{i_2} \cdots x^{i_n} }{\vert x\vert^{n}}\bigg].
\end{align*}
By expanding the functions $X^i$, $g^{ij}$ and $\Gamma_{ij}^k$ near the pole, by the previous equalities, we obtain
\begin{align*}
\mathcal{L}_X\Big(\frac{1}{\vert x\vert}\Big)&=\frac{1}{2\vert x\vert^2} X^{i_1}(o)\frac{x^{i_1}}{\vert x\vert}+\frac{1}{\vert x\vert}\Big[\widehat{A}_{i_1 i_2}\frac{x^{i_1}x^{i_2} }{\vert x\vert^2}+\widehat{B}_{i_1 \cdots i_4}\frac{x^{i_1}\cdots x^{i_4}}{\vert x\vert^4} \Big]\\
&\quad\,+\Big[\widehat{C}_{i_1 i_2 i_3}\frac{x^{i_1} x^{i_2}x^{i_3}}{\vert x\vert^3}+\widehat{D}_{i_1 \cdots i_5}\frac{x^{i_1}\cdots x^{i_5}}{\vert x\vert^5} \Big]\\
&\quad\, +\vert x\vert \Big[\widehat{E}_{i_1 \cdots i_4}\frac{x^{i_1}\cdots x^{i_4}}{\vert x\vert^4}+\widehat{F}_{i_1 \cdots i_6}\frac{x^{i_1}\cdots x^{i_6}}{\vert x\vert^6} \Big]\\
&\quad\,+ \text{a function that admits a $C^1$--extension on  $\overline{B}_{r_o}(o)$},
\end{align*}
where 
\begin{align*}
\widehat{A}_{i_1 i_2}&=\frac{1}{2}\Big[\,\partial_{x^{i_1}}\!X^{i_2}(o)+2\delta^{ij}\partial_{x^{i_1}} \!\Gamma^{i_2}_{ij}(o)-\delta_{ij}\partial_{x^{i_1}}\!\partial_{x^{i_2}}g^{ij}(o)\Big],\\
\widehat{B}_{i_1 \cdots i_4}&=\frac{3}{2}\partial_{x^{i_1}}\!\partial_{x^{i_2}}g^{i_3 i_4}(o)\,,\\
\widehat{C}_{i_1 i_2 i_3}&=\frac{1}{2}\Big[\,\frac{1}{2}\partial_{x^{i_1}}\!\partial_{x^{i_2}}X^{i_3}(o)+\delta^{ij}\partial_{x^{i_1}} \!\partial_{x^{i_2}}\!\Gamma^{i_3}_{ij}(o)-\frac{1}{3}\delta_{ij}\partial_{x^{i_1}}\!\partial_{x^{i_2}}\!\partial_{x^{i_3}}g^{ij}(o)\Big],\\
\widehat{D}_{i_1 \cdots i_5}&=\frac{1}{2}\partial_{x^{i_1}}\!\partial_{x^{i_2}}\!\partial_{x^{i_3}}g^{i_4 i_5}(o)\,,\\
\widehat{E}_{i_1 \cdots i_4}&=\frac{1}{2}\Big[\,\frac{1}{6}\partial_{x^{i_1}}\!\partial_{x^{i_2}}\!\partial_{x^{i_3}}X^{i_4}(o)+\frac{1}{3}\delta^{ij}\partial_{x^{i_1}} \!\partial_{x^{i_2}}\!\partial_{x^{i_3}}\!\Gamma^{i_4}_{ij}(o)+ \partial_{x^{i_1}}\!\partial_{x^{i_2}}g^{ij}(o)\,\partial_{x^{i_3}} \!\Gamma^{i_4}_{ij}(o)\\
&\qquad\ \ -\frac{1}{12}\delta_{ij}\partial_{x^{i_1}}\!\partial_{x^{i_2}}\!\partial_{x^{i_3}}\!\partial_{x^{i_4}}g^{ij}(o)\Big],\\
\widehat{F}_{i_1 \cdots i_6}&=\frac{1}{8}\partial_{x^{i_1}}\!\partial_{x^{i_2}}\!\partial_{x^{i_3}}\!\partial_{x^{i_4}}g^{i_5 i_6}(o)\,.
\end{align*}
\begin{align*}
\mathcal{L}_X\Big(b_{i_1}^{(0)}\frac{x^{i_1}}{\vert x\vert}\Big)&=-\frac{2}{\vert x\vert^2}b_{i_1}^{(0)}\frac{x^{i_1}}{\vert x\vert}-\frac{1}{2\vert x\vert}\Big[b_{i}^{(0)}X^i(o)-b_{i_1}^{(0)}X^{i_2}(o)\frac{x^{i_1}x^{i_2}}{\vert x\vert^2}\Big]\\
&\quad\,+\Big[ A^*_{i_1 }\frac{x^{i_1}}{\vert x\vert}+B^*_{i_1 i_2 i_3} \frac{x^{i_1} x^{i_2}x^{i_3}}{\vert x\vert^3}+C^*_{i_1 \cdots i_5}\frac{x^{i_1}\cdots x^{i_5}}{\vert x\vert^5}\Big]\\
&\quad\,+\vert x\vert\Big[ D^*_{i_1 i_2}\frac{x^{i_1}x^{i_2} }{\vert x\vert^2}+E^*_{i_1 \cdots i_4}\frac{x^{i_1}\cdots x^{i_4}}{\vert x\vert^4}+F^*_{i_1 \cdots i_6}\frac{x^{i_1}\cdots x^{i_6}}{\vert x\vert^6} \Big]\\
&\quad\,+ \text{a function that admits a $C^1$--extension on  $\overline{B}_{r_o}(o)$},
\end{align*}
where
\begin{align*}
A^*_{i_1 }&=-\Big[\,\frac{1}{2}b_i^{(0)}\partial_{x^{i_1}}\!X^{i}(o)+\delta^{ij}b_k^{(0)}\partial_{x^{i_1}} \!\Gamma^{k}_{ij}(o)\Big],\\
B^*_{i_1 i_2 i_3}&=\Big[\,\frac{1}{2}\partial_{x^{i_1}}\!X^{i_2}(o)+\delta^{ij}\partial_{x^{i_1}} \!\Gamma^{i_2}_{ij}(o)-\frac{1}{2}\delta_{ij}\partial_{x^{i_1}}\!\partial_{x^{i_2}}g^{ij}(o)\Big]b_{i_3}^{(0)}-b_i^{(0)}\partial_{x^{i_1}}\!\partial_{x^{i_2}}g^{i i_3}(o)\,,\\
C^*_{i_1 \cdots i_5}&=\frac{3}{2}\partial_{x^{i_1}}\!\partial_{x^{i_2}}g^{i_3 i_4}(o) b_{i_5}^{(0)}\,,\\
D^*_{i_1 i_2}&=-\frac{1}{2}\Big[\,\frac{1}{2}b_i^{(0)}\partial_{x^{i_1}} \!\partial_{x^{i_2}}X^i(o)+\delta^{ij}b_k^{(0)}\partial_{x^{i_1}} \!\partial_{x^{i_2}}\!\Gamma^{k}_{ij}(o) \Big],\\
E^*_{i_1 \cdots i_4}&=\Big[\, \frac{1}{4}\partial_{x^{i_1}}\!\partial_{x^{i_2}}X^{i_3}(o)+\frac{1}{2}\delta^{ij}\partial_{x^{i_1}} \!\partial_{x^{i_2}}\!\Gamma^{i_3}_{ij}(o)-\frac{1}{6}\delta_{ij}\partial_{x^{i_1}}\!\partial_{x^{i_2}}\!\partial_{x^{i_3}}g^{ij}(o)\Big]b_{i_4}^{(0)}\\
&\phantom{=\ }-\frac{1}{3}b_{i}^{(0)}\partial_{x^{i_1}}\!\partial_{x^{i_2}}\!\partial_{x^{i_3}}g^{i i_4}(o)\,,\\
F^*_{i_1 \cdots i_6}&=\frac{1}{2}\partial_{x^{i_1}}\!\partial_{x^{i_2}}\!\partial_{x^{i_3}}g^{i_4 i_5}(o)b_{i_6}^{(0)}\,.
\end{align*}
\begin{align*}
&\mathcal{L}_X \Big[\vert x\vert\Big(a^{(1)}+c_{i_1 i_2}^{(1)}\frac{x^{i_1} x^{i_2}}{\vert x\vert^2}+e_{i_1 i_2 i_3 i_4 }^{(1)}\frac{x^{i_1}x^{i_2} x^{i_3} x^{i_4}}{\vert x\vert^4}\Big) \Big] =\\
&\quad =\frac{1}{\vert x\vert}\Big[2\Big( a^{(1)}+\sum_{i=1}^{3}c_{i i}^{(1)} \Big)+\widetilde{A}_{i_1 i_2 }\frac{x^{i_1}x^{i_2}}{\vert x\vert^2}-18 e_{i_1 \cdots i_4 }^{(1)} \frac{x^{i_1} \cdots x^{i_4}}{\vert x\vert^4}\Big]\\
&\quad\quad\ -\frac{1}{2}\Big\{ \Big[a^{(1)}X^{i_1}(o)+X^i(o)\big( c_{ii_1}^{(1)}+c_{i_1i}^{(1)}\big)\Big]\frac{x^{i_1}}{\vert x\vert} +\widetilde{B}_{i_1 i_2 i_3 }\frac{x^{i_1}x^{i_2} x^{i_3}}{\vert x\vert^3}-3X^{i_1}(o) e_{i_2 \cdots i_5}^{(1)}\frac{x^{i_1}\cdots x^{i_5}}{\vert x\vert^5}\Big\}\\
&\quad\quad\ + \vert x\vert \Big[ \widetilde{C}_{i_1 i_2 } \frac{x^{i_1}x^{i_2}}{\vert x\vert^2} +\widetilde{D}_{i_1 \cdots i_4 }\frac{x^{i_1} \cdots x^{i_4}}{\vert x\vert^4}+\widetilde{E}_{i_1 \cdots i_6 }\frac{x^{i_1} \cdots x^{i_6}}{\vert x\vert^6}\\
&\qquad\qquad\quad+\frac{15}{2} \partial_{x^{i_1}}\!\partial_{x^{i_2}} g^{i_3 i_4}(o)e_{i_5 \cdots i_8}^{(1)}\frac{x^{i_1} \cdots x^{i_8}}{\vert x\vert^8} \Big]\\
&\quad\quad\ + \text{a function that admits a $C^1$--extension on  $\overline{B}_{r_o}(o)$},
\end{align*}
where 
\begin{align*}
\widetilde{A}_{i_1 i_2 }&=2\Big[-2c_{i_1 i_2}^{(1)}+e_{ ii i_1 i_2}^{(1)}+e_{ i i_1 i i_2}^{(1)}+ e_{ i i_1 i_2 i }^{(1)}+ e_{ i_1ii i_2}^{(1)}+e_{ i_1i i_2i}^{(1)}+e_{ i_1 i_2ii}^{(1)}\Big],\\
\widetilde{B}_{i_1 i_2 i_3 }&=- X^{i_1}(o)c_{i_2 i_3}^{(1)}+X^{i}(o)\big(e_{ i i_1 i_2 i_3}^{(1)}+e_{ i_1 i i_2 i_3}^{(1)}+e_{ i_1 i_2 i i_3}^{(1)}+e_{ i_1 i_2 i_3 i }^{(1)} \big),\\
\widetilde{C}_{i_1 i_2 }&=\frac{1}{2} \partial_{x^{i_1}}\!\partial_{x^{i_2}} g^{i j}(o)\big( \delta_{ij}a^{(1)}+2c_{ij}^{(1)}\big)-a^{(1)}\delta^{ij}\partial_{x^{i_1}} \!\Gamma^{i_2}_{ij}(o)-\delta^{ij}\partial_{x^{i_1}} \!\Gamma^{k}_{ij}(o)\big( c_{k i_2}^{(1)}+ c_{i_2 k}^{(1)}\big)\\
&\quad -\frac{1}{2} a^{(1)}\partial_{x^{i_1}} \!X^{i_2}(o)-\frac{1}{2}\partial_{x^{i_1}} \!X^{k}(o)\big( c_{k i_2}^{(1)}+ c_{i_2 k}^{(1)}\big),\\
\widetilde{D}_{i_1\cdots i_4}&=-\frac{1}{2}a^{(1)}\partial_{x^{i_1}}\!\partial_{x^{i_2}} g^{i_3 i_4}(o)-\partial_{x^{i_1}}\!\partial_{x^{i_2}} g^{i_3 i}(o)\big(c_{i i_4}^{(1)}+c_{i_4 i}^{(1)}\big)-\frac{1}{2}\delta_{ij} \partial_{x^{i_1}}\!\partial_{x^{i_2}} g^{i j}(o)c_{i_3 i_4}^{(1)}\\
&\quad +\partial_{x^{i_1}}\!\partial_{x^{i_2}} g^{i j}(o) \big(e_{ i j i_3 i_4}^{(1)}+e_{ i i_3 j i_4}^{(1)}+e_{ i i_3 i_4 j }^{(1)} +e_{ i_3 ij i_4}^{(1)} +e_{ i_3 i i_4 j }^{(1)}+e_{ i_3 i_4i j }^{(1)} \big)\\
&\quad +\delta^{ij}\partial_{x^{i_1}} \!\Gamma^{i_2}_{ij}(o)c_{i_3 i_4}^{(1)}-\delta^{ij}\partial_{x^{i_1}} \!\Gamma^{k}_{ij}(o)\big(e_{ k i_2 i_3 i_4}^{(1)}+e_{ i_2 k i_3 i_4}^{(1)}+ e_{ i_2 i_3 k i_4}^{(1)}+ e_{ i_2 i_3 i_4 k}^{(1)}\big)\\
&\quad +\frac{1}{2}\partial_{x^{i_1}} \!X^{i_2}(o)c_{i_3 i_4}^{(1)}-\frac{1}{2}\partial_{x^{i_1}} \!X^{k}(o)\big(e_{ k i_2 i_3 i_4}^{(1)}+e_{ i_2 k i_3 i_4}^{(1)}+ e_{ i_2 i_3 k i_4}^{(1)}+ e_{ i_2 i_3 i_4 k}^{(1)}\big),\\
\widetilde{E}_{i_1\cdots i_6}&=3\Big[\,\frac{1}{2}\partial_{x^{i_1}}\!\partial_{x^{i_2}} g^{i_3 i_4}(o)c_{i_5 i_6}^{(1)}-\partial_{x^{i_1}}\!\partial_{x^{i_2}} g^{ii_3}(o)\big(e_{ i i_4 i_5 i_6}^{(1)}+e_{ i_4 i i_5 i_6}^{(1)}+ e_{ i_4 i_5 i i_6}^{(1)}+ e_{ i_4 i_5 i_6 i}^{(1)}\big)\\
&\quad\quad\,-\frac{1}{2} \delta_{ij} \partial_{x^{i_1}}\!\partial_{x^{i_2}} g^{i j}(o)e_{i_3 \cdots i_6 }^{(1)}+ \delta_{ij}\partial_{x^{i_1}} \!\Gamma^{i_2}_{ij}(o)e_{i_3 \cdots i_6 }^{(1)}+\frac{1}{2}\partial_{x^{i_1}} \!X^{i_2}(o)e_{i_3 \cdots i_6 }^{(1)}\Big].
\end{align*}
\begin{align*}
&\mathcal{L}_X \Big[\vert x\vert^2\Big(b_{i_1}^{(2)}\,\frac{x^{i_1}}{\vert x\vert}+d_{i_1 i_2 i_3 }^{(2)}\frac{x^{i_1}x^{i_2} x^{i_3} }{\vert x\vert^3}+f_{i_1 i_2 i_3 i_4 i_5 }^{(2)}\frac{x^{i_1}x^{i_2} x^{i_3} x^{i_4} x^{i_5}}{\vert x\vert^5}\Big)\Big]=\\
&\quad =\Big[A^{\star}_{i_1 } \frac{x^{i_1}}{\vert x\vert}+B^{\star}_{i_1 i_2 i_3}\frac{x^{i_1} x^{i_2} x^{i_3}}{\vert x\vert^3}-24 f_{i_1 \cdots i_5 }^{(2)}\frac{x^{i_1} \cdots x^{i_5}}{\vert x\vert^5}\Big]\\
&\quad\quad+\vert x\vert\Big[ -\frac{1}{2}X^i(o)b_i^{(2)}+C^{\star}_{i_1 i_2} \frac{x^{i_1} x^{i_2}}{\vert x\vert^2} +D^{\star}_{i_1\cdots i_4}\frac{x^{i_1} \cdots x^{i_4}}{\vert x\vert^4}+\frac{3}{2}X^{i_1}(o) f_{i_2 \cdots i_6 }^{(2)}\frac{x^{i_1} \cdots x^{i_6}}{\vert x\vert^6}\Big]\\
&\quad\quad+ \text{a function that admits a $C^1$--extension on  $\overline{B}_{r_o}(o)$},
\end{align*}
where
\begin{align*}
A^{\star}_{i_1 }&=2\Big[\,2b_{i_1}^{(2)}+\sum_{i=1}^3 \big(d_{i_1\! ii}^{(2)}+d_{ i i_1\! i}^{(2)}+ d_{ii i_1}^{(2)}\big)\Big],\\
B^{\star}_{i_1 i_2 i_3}&= 2\Big[-3d_{i_1i_2 i_3}^{(2)} +\sum_{i=1}^3\big(f_{ii i_1 i_2 i_3}^{(2)}+f_{i i_1\! i i_2 i_3}^{(2)} +\cdots +f_{i_1 i_2 i_3 ii}^{(2)} \big)\Big],\\
C^{\star}_{i_1 i_2}&=-\frac{1}{2} \Big[X^{i_1}(o)b_{i_2}^{(2)}+X^i(o)\big(d_{i i_1 i_2}^{(2)}+d_{ i_1\! i i_2}^{(2)} + d_{ i_1 i_2 i }^{(2)} \big) \Big],\\
D^{\star}_{i_1\cdots i_4}&=-\frac{1}{2}\Big[-X^{i_1}(o) d_{i_2 i_3i_4}^{(2)}+X^i(o)\big(f_{i i_1i_2 i_3 i_4}^{(2)}+f_{i_1 i i_2 i_3 i_4}^{(2)}+\cdots + f_{i_1 i_2 i_3 i_4 i}^{(2)}\big)\Big].
\end{align*}
\begin{align*}
\mathcal{L}_X &\Big[\vert x\vert^3\Big(a^{(3)}+c_{i_1 i_2}^{(3)}\frac{x^{i_1} x^{i_2}}{\vert x\vert^2}+e_{i_1\cdots i_4 }^{(3)}\frac{x^{i_1}\cdots x^{i_4}}{\vert x\vert^4}+h_{i_1 \cdots i_6}^{(3)}\frac{x^{i_1}\cdots x^{i_6}}{\vert x\vert^6}
+l_{i_1 \cdots i_8}^{(3)}\frac{x^{i_1}\cdots x^{i_8}}{\vert x\vert^8}\Big) \Big]\\
&\quad =\vert x\vert\Big\{ 2\Big(6a^{(3)}+\sum_{i=1}^3 c_{ii}^{(3)}\Big)+A^{\bullet}_{i_1 i_2} \frac{x^{i_1} x^{i_2}}{\vert x\vert^2}+B^{\bullet}_{i_1 \cdots i_4}\frac{x^{i_1}\cdots x^{i_4}}{\vert x\vert^4}+C^{\bullet}_{i_1 \cdots i_6} \frac{x^{i_1}\cdots x^{i_6}}{\vert x\vert^6}\\
&\,\qquad\qquad-60 l_{i_1 \cdots i_8}^{(3)}\frac{x^{i_1}\cdots x^{i_8}}{\vert x\vert^8}\Big\}\\
&\quad\quad\ \,+ \text{a function that admits a $C^1$--extension on  $\overline{B}_{r_o}(o)$},
\end{align*}
where
\begin{align*}
A^{\bullet}_{i_1 i_2}&=2\Big[3c_{i_1 i_2}^{(3)}+\sum_{i=1}^3 \big(e_{ii i_1i_2}^{(3)}+e_{i i_1i i_2}^{(3)}+\cdots+e_{i_1i_2ii}^{(3)} \big)\Big],\\
B^{\bullet}_{i_1 \cdots i_4}&=2\Big[-4 e_{i_1\cdots i_4 }^{(3)}+\sum_{i=1}^3 \big(h_{ii i_1i_2i_3i_4}^{(3)}+h_{i i_1 i i_2i_3i_4}^{(3)}+\cdots+h_{i_1i_2i_3i_4ii}^{(3)} \big) \Big] ,\\
C^{\bullet}_{i_1 \cdots i_6}&=2\Big[-15 h_{i_1 \cdots i_6}^{(3)}+\sum_{i=1}^3 \big(l_{ii i_1i_2 \cdots i_6}^{(3)}+l_{i i_1 i i_2 \cdots i_6}^{(3)}+\cdots+l_{i_1 \cdots i_6 ii}^{(3)} \big) \Big].
\end{align*}
Therefore, by combining all the previous computations, we obtain
\begin{align*}
&\mathcal{L}_X \Big[\frac{1}{\vert x\vert}+b_{i_1}^{(0)}\frac{x^{i_1}}{\vert x\vert}+\vert x\vert\Big(a^{(1)}+c_{i_1 i_2}^{(1)}\frac{x^{i_1} x^{i_2}}{\vert x\vert^2}+e_{i_1 \cdots i_4 }^{(1)}\frac{x^{i_1}\cdots x^{i_4}}{\vert x\vert^4}\Big)\\
&\quad\quad+\vert x\vert^2\Big(b_{i_1}^{(2)}\,\frac{x^{i_1}}{\vert x\vert}+d_{i_1 i_2 i_3 }^{(2)}\frac{x^{i_1}x^{i_2} x^{i_3} }{\vert x\vert^3}+f_{i_1\cdots i_5 }^{(2)}\frac{x^{i_1}\cdots x^{i_5}}{\vert x\vert^5}\Big)\\
&\quad\quad +\vert x\vert^3\Big(a^{(3)}+c_{i_1 i_2}^{(3)}\frac{x^{i_1} x^{i_2}}{\vert x\vert^2}+e_{i_1\cdots i_4 }^{(3)}\frac{x^{i_1}\cdots x^{i_4}}{\vert x\vert^4}+h_{i_1 \cdots i_6}^{(3)}\frac{x^{i_1}\cdots x^{i_6}}{\vert x\vert^6}
+l_{i_1\cdots i_8}^{(3)}\frac{x^{i_1}\cdots x^{i_8}}{\vert x\vert^8}\Big)\Big]\\
&\quad =\frac{1}{\vert x\vert^2} \Big(\frac{1}{2}X^{i_1}(o)- 2 b_{i_1}^{(0)}\Big)\frac{x^{i_1}}{\vert x\vert}\\
&\quad\quad\ +\frac{1}{\vert x\vert}\Big\{2\Big( a^{(1)}+\sum_{i=1}^{3}c_{i i}^{(1)} \Big) -\frac{1}{2}b_{i}^{(0)}X^i(o)+\big(\widehat{A}_{i_1 i_2}+\frac{1}{2}b_{i_1}^{(0)}X^{i_2}(o)+\widetilde{A}_{i_1 i_2 }\big)\frac{x^{i_1}x^{i_2} }{\vert x\vert^2}\\
&\quad\quad\quad\quad\quad\quad+\big(\widehat{B}_{i_1 \cdots i_4}-18 e_{i_1 \cdots i_4 }^{(1)}\big)\frac{x^{i_1}\cdots x^{i_4}}{\vert x\vert^4}\Big\}\\
&\quad\quad\ +\Big\{\Big[-\frac{1}{2} a^{(1)}X^{i_1}(o)-\frac{1}{2}X^i(o)\big( c_{ii_1}^{(1)}+c_{i_1i}^{(1)}\big)+A^*_{i_1 }+A^{\star}_{i_1 }\Big]\frac{x^{i_1}}{\vert x\vert}\\
&\quad\quad\quad\quad+\big(\widehat{C}_{i_1 i_2 i_3}+B^*_{i_1 i_2 i_3}+\widetilde{B}_{i_1 i_2 i_3 }+B^{\star}_{i_1 i_2 i_3}\big)\frac{x^{i_1} x^{i_2}x^{i_3}}{\vert x\vert^3}\\
&\quad\quad\quad\quad+\big(\widehat{D}_{i_1 \cdots i_5}+C^*_{i_1 \cdots i_5}-3X^{i_1}(o) e_{i_2 \cdots i_5}^{(1)}-24 f_{i_1 \cdots i_5 }^{(2)}\big)\frac{x^{i_1}\cdots x^{i_5}}{\vert x\vert^5}\Big\}\\
&\quad\quad\ +\vert x\vert\Big\{2\Big(6a^{(3)}+\sum_{i=1}^3 c_{ii}^{(3)}\Big) -\frac{1}{2}X^i(o)b_i^{(2)}+\big(D^*_{i_1 i_2}+\widetilde{C}_{i_1 i_2 }+C^{\star}_{i_1 i_2 }+A^{\bullet}_{i_1 i_2 }\big)\frac{x^{i_1} x^{i_2}}{\vert x\vert^2}\\
&\quad\quad\quad\quad\quad+\big(\widehat{E}_{i_1 \cdots i_4}+E^*_{i_1 \cdots i_4}+\widetilde{D}_{i_1 \cdots i_4 }+D^{\star}_{i_1\cdots i_4}+B^{\bullet}_{i_1 \cdots i_4}\big)\frac{x^{i_1} \cdots x^{i_4}}{\vert x\vert^4}\\
&\quad\quad\quad\quad\quad+\big(\widehat{F}_{i_1 \cdots i_6}+F^*_{i_1 \cdots i_6}+\widetilde{E}_{i_1 \cdots i_6 }+\frac{3}{2}X^{i_1}(o) f_{i_2 \cdots i_6 }^{(2)}+C^{\bullet}_{i_1 \cdots i_6}\big)\frac{x^{i_1}\cdots x^{i_6}}{\vert x\vert^6} \\
&\quad\quad\quad\quad\quad+\Big(\frac{15}{2} \partial_{x^{i_1}}\!\partial_{x^{i_2}} g^{i_3 i_4}(o)e_{i_5 \cdots i_8}^{(1)}-60 l_{i_1 \cdots i_8}^{(3)}\Big)\frac{x^{i_1}\cdots x^{i_8}}{\vert x\vert^8}\Big\}\\
&\quad\quad\ +\text{a function that admits a $C^1$--extension $h$ on $\overline{B}_{r_o}(o)$}\,.
\end{align*}
Then, we choose constants $b_{i_1}^{(0)}$, $a^{(1)}$, $c_{i_1 i_2}^{(1)}$, $e_{i_1 i_2 i_3 i_4 }^{(1)}$, $b_{i_1}^{(2)}$, $d_{i_1 i_2 i_3 }^{(2)}$, $f_{i_1 i_2 i_3 i_4 i_5 }^{(2)}$, $a^{(3)}$, $c_{i_1 i_2}^{(3)}$, $e_{i_1 i_2 i_3 i_4 }^{(3)}$, $h_{i_1 i_2 i_3 i_4 i_5 i_6}^{(3)}$, $l_{i_1 i_2 i_3 i_4 i_5 i_6 i_7 i_8}^{(3)}$ in the following way:
\begin{align*}
b_{i_1}^{(0)}&=\frac{1}{4}X^{i_1}(o)\,,\\
e_{i_1 \cdots i_4 }^{(1)}&=\frac{1}{18}\widehat{B}_{i_1 \cdots i_4}\,,\\
c_{i_1 i_2}^{(1)}&=\frac{1}{4}\Big[2\big(e_{ ii i_1 i_2}^{(1)}+e_{ i i_1 i i_2}^{(1)}+ e_{ i i_1 i_2 i }^{(1)}+ e_{ i_1ii i_2}^{(1)}+e_{ i_1i i_2i}^{(1)}+e_{ i_1 i_2ii}^{(1)}\big)+\widehat{A}_{i_1 i_2}+\frac{1}{2}b_{i_1}^{(0)}X^{i_2}(o)\Big],\\
a^{(1)}&=-\sum_{i=1}^{3}c_{i i}^{(1)}+\frac{1}{4}b_{i}^{(0)}X^i(o)\,,\\
\end{align*}
\begin{align*}
f_{i_1 i_2 i_3 i_4 i_5 }^{(2)}&=\frac{1}{24}\Big(\widehat{D}_{i_1 \cdots i_5}+C^*_{i_1 \cdots i_5}-3X^{i_1}(o) e_{i_2 \cdots i_5}^{(1)}\Big),\\
d_{i_1 i_2 i_3 }^{(2)}&=\frac{1}{6}\Big[2\sum_{i=1}^3\big(f_{ii i_1 i_2 i_3}^{(2)}+f_{i i_1\! i i_2 i_3}^{(2)} +\cdots +f_{i_1 i_2 i_3 ii}^{(2)} \big)+\widehat{C}_{i_1 i_2 i_3}+B^*_{i_1 i_2 i_3}+\widetilde{B}_{i_1 i_2 i_3 }\Big],\\
b_{i_1}^{(2)}&=\frac{1}{4}\Big[\frac{1}{2} a^{(1)}X^{i_1}(o)+\frac{1}{2}X^i(o)\big( c_{ii_1}^{(1)}+c_{i_1i}^{(1)}\big)-A^*_{i_1 }-\,2\sum_{i=1}^3 \big(d_{i_1\! ii}^{(2)}+d_{ i i_1\! i}^{(2)}+ d_{ii i_1}^{(2)}\big)\Big],\\
l_{i_1 i_2 i_3 i_4 i_5 i_6 i_7 i_8}^{(3)}&=\frac{1}{8} \partial_{x^{i_1}}\!\partial_{x^{i_2}} g^{i_3 i_4}(o)e_{i_5 \cdots i_8}^{(1)}\,,\\
h_{i_1 i_2 i_3 i_4 i_5 i_6}^{(3)}&=\frac{1}{30}\Big[\widehat{F}_{i_1 \cdots i_6}+F^*_{i_1 \cdots i_6}+\widetilde{E}_{i_1 \cdots i_6 }+\frac{3}{2}X^{i_1}(o) f_{i_2 \cdots i_6 }^{(2)}\\
&\qquad\quad+2\sum_{i=1}^3 \big(l_{ii i_1i_2 \cdots i_6}^{(3)}+l_{i i_1 i i_2 \cdots i_6}^{(3)}+\cdots+l_{i_1 \cdots i_6 ii}^{(3)} \big) \Big],\\
e_{i_1 i_2 i_3 i_4 }^{(3)}&=\frac{1}{8}\Big[\widehat{E}_{i_1 \cdots i_4}+E^*_{i_1 \cdots i_4}+\widetilde{D}_{i_1 \cdots i_4 }+D^{\star}_{i_1\cdots i_4}\\
&\qquad\quad+2\sum_{i=1}^3 \big(h_{ii i_1i_2i_3i_4}^{(3)}+h_{i i_1 i i_2i_3i_4}^{(3)}+\cdots+h_{i_1i_2i_3i_4ii}^{(3)} \big) \Big], \\
c_{i_1 i_2}^{(3)}&=-\frac{1}{6}\Big[D^*_{i_1 i_2}+\widetilde{C}_{i_1 i_2 }+C^{\star}_{i_1 i_2 }+2\sum_{i=1}^3 \big(e_{ii i_1i_2}^{(3)}+e_{i i_1i i_2}^{(3)}+\cdots+e_{i_1i_2ii}^{(3)} \big)\Big],\\
a^{(3)}&=\frac{1}{24}X^i(o)b_i^{(2)}-\frac{1}{6}\sum_{i=1}^3 c_{ii}^{(3)}\,.
\end{align*}

\noindent
\textbf{{Acknowledgements.}} {\em All authors are members of INdAM--GNAMPA. The first author is partially supported by the PRIN Project 2022E9CF89 ``GEPSO -- Geometric Evolution Problems and Shape Optimization''.}

\bibliographystyle{amsplain}
\bibliography{biblio}

\end{document}